%% file: main.tex
\documentclass{article}

\usepackage[colorlinks,linkcolor=blue,citecolor=blue, pagebackref=true,linktocpage]{hyperref}
\usepackage[margin=1.1in]{geometry}
\usepackage[numbers]{natbib}
\usepackage{booktabs}

\usepackage{amsmath, amsthm, amssymb, amsfonts}
\usepackage{thmtools}
\usepackage{thm-restate}
\input{preamble}

\theoremstyle{plain}

\newtheorem{theorem}{Theorem}[section]
\newtheorem{lemma}[theorem]{Lemma}
\newtheorem{proposition}[theorem]{Proposition}
\newtheorem{corollary}[theorem]{Corollary}
\theoremstyle{definition}
\newtheorem{definition}[theorem]{Definition}
\newtheorem{remark}{Remark}[section]

\usepackage{authblk}

\title{Nonasymptotic and distribution-uniform\\Koml\'os--Major--Tusn\'ady approximation\\\ifverbose \textcolor{red}{(Verbose version)} \fi}
\author{
  Ian Waudby-Smith$^\diamond$, Martin Larsson$^\dagger$, and Aaditya Ramdas$^\ddagger$\\
  $^\diamond$Department of Statistics and The Miller Institute, University of California, Berkeley\\
  $^\dagger$Department of Mathematical Sciences, Carnegie Mellon University\\
  $^\ddagger$Department of Statistics, Stanford University
}
\date{}

\begin{document}
\maketitle
\setcounter{tocdepth}{2}
\makeatletter
\renewcommand\tableofcontents{%
  \@starttoc{toc}%
}

\makeatother

\begin{abstract}
  \input{abstract}

\ifverbose
\ianws{Note to Aaditya and Martin: You can use \texttt{\textbackslash aaditya\{\}} and \texttt{\textbackslash martin\{\}} to leave in-line comments.
Also, you can replace \texttt{\textbackslash verbosetrue} with \texttt{\textbackslash verbosefalse} in \texttt{preamble.tex} to toggle verbose derivations.}
\fi
\end{abstract}

\input{content}

\subsection*{Acknowledgments}

The authors thank Tudor Manole, Sivaraman Balakrishnan, and Arun Kuchibhotla for helpful discussions. The authors acknowledge the use of AI tools to find errors and inconsistencies in mathematical derivations. No prose or proofs were generated using these tools.

\bibliographystyle{plainnat}
\bibliography{references.bib}
\newpage
\appendix

\input{appendix}

\end{document}

%% file: preamble.tex
\usepackage[utf8]{inputenc}
\usepackage{dsfont}
\usepackage{graphicx}
\usepackage[makeroom]{cancel}
\usepackage[shortlabels]{enumitem}
\usepackage{mathabx}
\usepackage[capitalise,noabbrev]{cleveref}
\usepackage{xcolor}
\usepackage{array}
\usepackage{colortbl}
\usepackage{tcolorbox}
\usepackage{natbib}
\usepackage{accents}
\usepackage{mathrsfs}
\usepackage{tikz}
\usepackage{autonum}
\usetikzlibrary{positioning, arrows.meta}

\usepackage[normalem]{ulem}

\newcommand{\RR}{\mathbb R}
\newcommand{\EE}{\mathbb E}
\newcommand{\NN}{\mathbb N}

\newcommand{\Acal}{\mathcal A}

\newcommand{\Dcal}{\mathcal D}
\newcommand{\Fcal}{\mathcal F}

\newcommand{\Ical}{\mathcal I}

\newcommand{\Ncal}{\mathcal N}

\newcommand{\Pcal}{\mathcal P}

\newcommand{\1}{\mathds{1}}

\newcommand{\seq}[4]{(#1)_{{#2}={#3}}^{#4}}

\newcommand{\infseqn}[1]{\seq{#1}{n}{1}{\infty}}
\newcommand{\infseqm}[1]{\seq{#1}{m}{1}{\infty}}

\newcommand{\dd}{\mathrm{d}}

\newcommand{\oAcalas}{\widebar o_\Acal}%

\newcommand{\iid}{\text{i.i.d.}}

\newcommand{\eps}{\varepsilon}
\newcommand{\Var}{\mathrm{Var}}

\newcommand{\ubar}[1]{\underaccent{\bar}{#1}}

\defcitealias{komlos1975approximation}{KMT-I}
\defcitealias{komlos1976approximation}{KMT-II}

\newcommand{\supP}{{\sup_{P \in \Pcal}}}
\newcommand{\supalpha}{{\sup_{\alpha \in \Acal}\ }}
\newcommand{\Pin}{{P \in \Pcal}}
\newcommand{\alphain}{{\alpha \in \Acal}}
\newcommand{\Pstar}{{P^\star}}
\newcommand{\supkm}{{\sup_{k \geq m}}}
\newcommand{\supknm}{{\sup_{k,n \geq m}}}

\newcommand{\tth}{{\text{th}}}
\newcommand{\mto}{{m \to \infty}}
\newcommand{\nto}{{n \to \infty}}
\newcommand{\Mto}{{M \to \infty}}
\newcommand{\Kto}{{K \to \infty}}

\newcommand{\geqn}{{\geq n}}
\newcommand{\geqm}{{\geq m}}

\newcommand{\maxkn}{{\max_{1 \leq k \leq n}}}

\newcommand{\brackn}{{(n)}}
\newcommand{\brackb}{{(b)}}

\newcommand{\bracki}{{(i)}}

\newcommand{\brackleq}{{(\leq)}}
\newcommand{\brackalpha}{{(\alpha)}}

\newcommand{\Ptil}{{\widetilde P}}

\newcommand{\Xtil}{\widetilde X}
\newcommand{\Ytil}{\widetilde Y}

\newcommand{\Lambdatil}{\widetilde \Lambda}

\newcommand{\Palpha}{{P_\alpha}}
\newcommand{\Palphatil}{{\widetilde P_\alpha}}
\newcommand{\Palphadot}{{P_\alphadot}}
\newcommand{\Palphadottil}{{\widetilde P_\alphadot}}

\newcommand{\alphadot}{{\dot\alpha}}

\newcommand{\probspace}{(\Omega, \Fcal, P)}
\newcommand{\probspacealpha}{(\Omega_\alpha, \Fcal_\alpha, P_\alpha)}
\newcommand{\probspacealphatil}{(\widetilde \Omega_\alpha, \widetilde \Fcal_\alpha, \widetilde P_\alpha)}
\newcommand{\probspaces}{(\Omega_\alpha, \Fcal_\alpha, P_{\alpha})_{\alpha \in \Acal}}
\newcommand{\probspacetilde}{(\widetilde \Omega, \widetilde \Fcal, \widetilde P)}
\newcommand{\probspacestilde}{(\widetilde \Omega_\alpha, \widetilde \Fcal_\alpha, \widetilde P_\alpha)_{\alpha \in \Acal}}

\newcommand{\mk}{{m,k}}

\newcommand{\Nbmk}{{\ubar N_{b(m)}, k}}
\newcommand{\mj}{{m,j}}

\newcommand{\Max}{\mathrm{max}}

\newcommand{\bmax}{{b, \Max}}
\newif\ifverbose %
\verbosefalse %

\newcommand{\verbose}[1]{\ifverbose\textcolor{gray}{\textit{verbose $\rightarrow$}}\quad #1\textcolor{gray}{\textit{non-verbose $\rightarrow$}}\quad \fi}

%% file: abstract.tex
  We present nonasymptotic concentration inequalities for sums of independent and identically distributed random variables that yield asymptotic strong Gaussian approximations of Koml\'os, Major, and Tusn\'ady (KMT) [1975,1976]. The constants appearing in our inequalities are either universal or explicit, and thus as corollaries, they imply distribution-uniform generalizations of the aforementioned KMT approximations. In particular, it is shown that uniform integrability of a random variable's $q^{\text{th}}$ moment is both necessary and sufficient for the KMT approximations to hold uniformly at the rate of $o(n^{1/q})$ for $q > 2$ and that having a uniformly lower bounded Sakhanenko parameter --- equivalently, a uniformly upper-bounded Bernstein parameter --- is both necessary and sufficient for the KMT approximations to hold uniformly at the rate of $O(\log n)$. Instantiating these uniform results for a single probability space yields the analogous results of KMT exactly.

%% file: content.tex
\section{Introduction}\label{section:introduction}

Let $\infseqn{X_n}$ be an infinite sequence of independent and identically distributed (\iid{}) random variables on a probability space $(\Omega,\Fcal, P)$ such that $\EE_P(X) = 0$ and $0 < \Var_P(X) < \infty$. In a seminal 1964 paper, \citet{strassen1964invariance} initiated the study of strong Gaussian approximations by showing that there exists a coupling between the partial sum $S_n := \sum_{i=1}^n X_i$ and another partial sum $G_n := \sum_{i=1}^n Y_i$ consisting of \iid{} $\infseqn{Y_n}$ which are marginally Gaussian with mean zero and variance $\Var_P(X)$ such that
\begin{equation}\label{eq:intro-strassen64}
  |S_n - G_n| = o \left (\sqrt{n \log \log n} \right )\quad \text{$P$-almost surely.}\footnote{In order to say that this approximation holds $P$-almost surely, the probability space may need to be enlarged to encompass the joint distribution of $X$ and $Y$. This technicality is made formal in \cref{section:preliminaries}.}
\end{equation}
In his 1967 paper, \citet{strassen1967and} later showed that if in addition $\EE_P|X|^4 < \infty$, then the above coupling rate can be improved to
$O \left ( n^{1/4} (\log n)^{1/2} (\log \log n)^{1/4} \right )$.
While the former estimate in \eqref{eq:intro-strassen64} cannot be improved without further assumptions, the latter can. The papers of \citet*{komlos1975approximation,komlos1976approximation} and \citet{major1976approximation} when taken together show the remarkable fact that if $\EE_P|X|^q < \infty$ for $q > 2$, the rate in \eqref{eq:intro-strassen64} can be improved to
\begin{equation}\label{eq:intro-kmt}
  |S_n - G_n| = o(n^{1/q}) \quad\text{$P$-almost surely,}
\end{equation}
and if $X$ has a finite moment generating function, then $o(n^{1/q})$ can be replaced by $O(\log n)$. Furthermore, \citet{komlos1975approximation,komlos1976approximation} showed that their bounds are unimprovable without further assumptions. These strong Gaussian approximations (sometimes called ``strong invariance principles'' or ``strong embeddings'') can be viewed as almost-sure characterizations of central limit theorem-type behavior via explicit couplings. 

Recall that the central limit theorem --- a statement about convergence in distribution rather than couplings --- can be stated \emph{distribution-uniformly}, a fact that can be deduced from the Berry--Esseen theorem \citep{berry1941accuracy,esseen1942liapunov} and its extension due to \citet{katz1963note}. Taken together, these authors showed that there exists a universal constant $C_\mathrm{BE}$ so that for $q \in (2,3]$,
\begin{equation}\label{eq:intro-clt}
  \sup_{x \in \RR} \left \lvert  P \left ( \frac{S_n}{\sqrt{n \Var_P(X)}} \leq x \right ) - \Phi(x) \right \rvert \leq  \frac{C_\mathrm{BE}\EE_P|X|^q}{\Var_P(X)^{q/2} n^{q/2-1}},
\end{equation}
where $\Phi$ is the cumulative distribution function of a standard Gaussian.
Taking $\Pcal$ to be a collection of distributions for which $\sup_\Pin \EE_P|X|^q < \infty$ and $\inf_{\Pin} \Var_P(X) > 0$ and taking limits as $n \to \infty$ yields a uniform central limit theorem.
The $\Pcal$-uniform convergence implied by \eqref{eq:intro-clt} clearly implies the usual $P$-pointwise notion, but it also highlights that convergence to the Gaussian law is insensitive to arbitrary distributional perturbations within $\Pcal$.

These discussions surrounding distribution-uniformity of the central limit theorem raise the question: is there a sense in which the Koml\'os-Major-Tusn\'ady approximations are distribution-uniform, and if so, for what class of distributions? In order to answer this question formally, we will need to make certain notions of ``strongly distribution-uniform'' and ``distribution-uniformly coupled processes'' precise. These definitions will be introduced soon, but they will always recover the usual single-distribution (``pointwise'') notions such as in \eqref{eq:intro-kmt} when $\Pcal = \{P\}$ is taken to be a singleton. Nevertheless, let us foreshadow the results to come by stating them informally. 

In \cref{theorem:uniform-kmt-power-moments}, we will show that if $\Pcal$ is a collection of distributions for which the variance is uniformly positive and the $q^\tth$ moment is uniformly integrable for some $q > 2$, meaning that
\begin{equation}\label{eq:intro-uniform-integrability}
  \lim_\mto \supP \EE_P \left ( |X|^q \1 \{ |X|^q > m \} \right ) = 0,
\end{equation}
then the strong approximation in \eqref{eq:intro-kmt} holds uniformly in $\Pcal$. Note that in the $P$-pointwise case where $\Pcal = \{ P \}$ is a singleton, \eqref{eq:intro-uniform-integrability} reduces to integrability of the $q^\tth$ moment, i.e.~$\EE_P |X|^q < \infty$ and hence \eqref{eq:intro-uniform-integrability} imposes no additional assumptions in that case. Moreover, we will show that \eqref{eq:intro-uniform-integrability} is both sufficient and necessary for \eqref{eq:intro-kmt} to hold uniformly for random variables with uniformly bounded variances. 

In \cref{theorem:uniform-kmt-exponential}, we show that if $X$ has a uniformly positive variance and a uniformly upper-bounded moment generating function (equivalently, uniformly upper-bounded \emph{Bernstein} or uniformly positive \emph{Sakhanenko} parameters; see \cref{proposition:sakhanenko-param-bernstein-uniform-integrability} for details), then the same approximation holds uniformly at an improved rate of $O( \log n )$ and with a constant that explicitly scales with these so-called Bernstein and Sakhanenko parameters in ways that will be made precise in \cref{section:exponential-moments}. Furthermore, we show that these uniform boundedness conditions on the moment generating function or Bernstein and Sakhanenko parameters are once again both necessary and sufficient under uniformly bounded variance assumptions.

All of our derivations of distribution-uniform strong approximations in \cref{section:exponential-moments,section:power-moments} will be accompanied by a corresponding nonasymptotic concentration inequality for $|S_k - G_k|$ uniformly in $k$. These nonasymptotic inequalities will serve as central technical devices in the proofs, but they may be of interest in their own right.
One key challenge in deriving distribution-uniform strong approximations (and especially the concentration inequalities on which they rely), is that it is no longer obvious how one can invoke qualitative proof techniques --- in the sense of a limiting property occurring \emph{eventually} with no quantification of when --- that are ubiquitous in the classical literature on strong approximations such as the Borel-Cantelli lemmas, zero-one laws, the Khintchine--Kolmogorov convergence theorem, and Kronecker's lemma. In a certain sense, our arguments replace these invocations with analogues that provide quantitative control of inequalities simultaneously across probability spaces. The replacement of qualitative limit theorems with explicit inequalities underlies much of this work.

\subsection{Preliminaries, notation, and uniform constructions}\label{section:preliminaries}
Let us give a brief refresher on the typical setting considered in the classical literature on strong approximations which will serve as relevant context for the uniform generalizations to come. When presenting strong approximations, one typically starts with a sequence of random variables $\infseqn{X_n}$ on a probability space $\probspace$. A new probability space $\probspacetilde$ is then constructed, containing two sequences $\infseqn{\Xtil_n}$ and $\infseqn{\Ytil_n}$ wherein the former is equidistributed with $\infseqn{X_n}$ while the latter is a sequence of marginally independent Gaussians with the property that
\begin{equation}
  \sum_{k=1}^n \Xtil_k - \sum_{k=1}^n \Ytil_k = O(r_n), \quad \text{$\Ptil$-almost surely},
\end{equation}
for some monotone real sequence $\infseqn{r_n}$ such as $n^{1/q}$ in \eqref{eq:intro-kmt}. This procedure of constructing a new probability space containing the relevant sequences is often given the name of ``the construction'' and authors will often start with a sequence $\infseqn{X_n}$ on $\probspace$ and say that ``a construction exists'' and then state the result in terms of $\infseqn{\widetilde X_n}$ and $\infseqn{\widetilde Y_n}$ without explicitly describing the new probability space for the sake of brevity. 

Here, we make precise the notion of ``a collection of constructions'' so that statements about uniform strong approximations can be made. The notion of ``uniformity'' will be with respect to an arbitrary index set $\Acal$, and we will consider $\Acal$-indexed probability spaces $\probspaces$ each containing a sequence of random variables $\infseqn{X^\brackalpha_n};\ \alpha \in \Acal$. For each $\alpha \in \Acal$, there will be a corresponding construction $(\widetilde \Omega_\alpha, \widetilde \Fcal_\alpha, \widetilde P_\alpha)$ containing both
\begin{equation}
  \infseqn{\Xtil_n^\brackalpha} \quad \text{and} \quad\infseqn{\Ytil_n^\brackalpha}
\end{equation}
so that $\infseqn{\Xtil_n^\brackalpha}$ is equidistributed with $\infseqn{X_n^\brackalpha}$ and $\infseqn{\Ytil_n^\brackalpha}$ is a sequence of marginally independent Gaussians, similar to the pointwise (non-uniform) setting. We refer to $\probspacestilde$ as the ``collection of constructions''. Note that we have yet to define the sense in which the difference $\sum_{k=1}^n \Xtil_k^\brackalpha - \sum_{k=1}^n \Ytil_k^\brackalpha$ vanishes uniformly in $\Acal$, but these details will be provided later in \cref{section:exponential-moments,section:power-moments}. Throughout the paper, we will typically drop the superscript  $(\alpha)$ when the particular index $\alphain$ is clear from context.\footnote{Let us briefly justify our use of an arbitrary index set rather than directly using a collection of probability measures on a common measurable space despite the latter being notationally simpler; see e.g.~\citep{waudby2024distribution}. In several instances throughout the proofs, we will rely on some existing strong coupling results that begin by considering random variables on a probability space $\probspace$ from which a new space $\probspacetilde$ is constructed. In general, however, the new measurable space $(\widetilde \Omega, \widetilde \Fcal)$ may depend on the original probability space in nontrivial ways. As such, when we invoke the aforementioned results, we treat them as black boxes so that there is one construction $\probspacealphatil$ for each initial space $\probspacealpha$ indexed by an arbitrary set $\alphain$.}

Finally, in the case where $\infseqn{X_n}$ are \iid{}, we may discuss the distributional properties of the random variable $X$ which will always be an independent copy of $X_1$.

\subsection{Outline and summary of contributions}
We begin in \cref{section:exponential-moments} by recalling the pointwise strong approximation result of \citet{komlos1976approximation} which says that $\sum_{k=1}^n (\Xtil_k - \Ytil_k) = O(\log n)$ with probability one when $X$ has a finite exponential moment. We introduce the notion of \emph{Sakhanenko regularity} (indexed by a parameter $\ubar \lambda > 0$) as a distribution-uniform generalization of exponential moments being finite and in \cref{theorem:uniform-kmt-exponential} we show that this particular notion of regularity is both necessary and sufficient for strong approximations to hold uniformly at the same rate, and inversely with $\ubar \lambda$. In \cref{theorem:concentration-kmt} we present a time-uniform concentration inequality for $|\sum_{i=1}^k(\Xtil_i - \Ytil_i)|$ for any $m \geq 5$ that yields the former uniform approximation as a near-immediate corollary. \cref{section:power-moments} follows a similar structure where it is first shown in \cref{theorem:uniform-kmt-power-moments} that uniform integrability of the $q^\tth$ moment (as displayed in \eqref{eq:intro-uniform-integrability}) is equivalent to the strong approximation of \eqref{eq:intro-kmt} holding uniformly, i.e.~at the rate of $o(n^{1/q});\ q > 2$. We also show that it is a downstream consequence of a nonasymptotic time-uniform concentration inequality for $|\sum_{i=1}^k(\Xtil_i - \Ytil_i)|$ under weaker independent but non-\iid{} conditions on $\infseqn{X_n}$. In addition, the proof of \cref{theorem:uniform-kmt-power-moments} involves applications of certain new distribution-uniform generalizations~\citep{waudby2024distribution} of familiar strong convergence results such as Kronecker's lemma and the Khintchine--Kolmogorov theorem. \cref{section:proofs-exponential-moments,section:proofs-power-moments} contain proofs of results in \cref{section:exponential-moments,section:power-moments}, respectively.

\section{Random variables with finite exponential moments}\label{section:exponential-moments}

In this section, we focus on \iid{} mean-zero random variables $\infseqn{X_n}$ on the collection of probability spaces $\probspaces$ so that $X$ has a finite exponential moment, meaning that
\begin{equation}
  \EE_\Palpha \exp \left \{ t_\alpha |X| \right \} < \infty
\end{equation}
for some $t_\alpha > 0; \alpha \in \Acal$. Throughout, when we discuss a construction (or an $\Acal$-indexed collection thereof), we will use $\Lambdatil_n$ as shorthand for
\begin{equation}
  \Lambdatil_n := \sum_{i=1}^n \widetilde X_i - \sum_{i=1}^n \widetilde Y_i,\quad n \in \NN,
\end{equation}
where $\infseqn{\widetilde X_n}$ will be taken to be equidistributed with $\infseqn{X_n}$, and $\infseqn{\widetilde Y_n}$ will be marginally \iid{} Gaussian random variables with mean zero and $\Var_{\Palphatil}(\widetilde Y) = \Var_\Palpha(X)$ for each $\alpha \in \Acal$ as alluded to in \cref{section:preliminaries}.

The central results of \citet*{komlos1975approximation,komlos1976approximation} are distribution-pointwise statements, which in the context of our setting means that $\Acal = \{\alphadot\}$ is taken to be a singleton. Focusing on the case of finite exponential moments, one of their central results states that given \iid{} random variables $\infseqn{X_n}$ on the probability space $\probspace \equiv (\Omega_\alphadot, \Fcal_\alphadot, P_\alphadot)$, if $\EE_\Palphadot \exp \{t |X| \} < \infty$ for some $t > 0$ then there exists a construction $(\widetilde \Omega, \widetilde \Fcal, \widetilde P)$ so that for some constant $K(P)$,
\begin{equation}\label{eq:kmt-pointwise-exponential}
  \Ptil \left ( \limsup_{n \to \infty} \frac{|\Lambdatil_n|}{\log n} \leq K(P) \right ) = 1,
\end{equation}
and hence $\Lambdatil_n = O(\log n)$ with $\Ptil$-probability one.
One key challenge in obtaining a distribution-uniform analogue of the above is that the constant $K(P)$ depends on the distribution $P$ in a way that is not immediately transparent. 
It is not just the final statement given in \eqref{eq:kmt-pointwise-exponential} that depends on distribution-dependent constants, but a crucial high-probability bound on $\max_{1\leq k\leq n}|  \Lambdatil_k |$ for any $n$. Indeed, to arrive at \eqref{eq:kmt-pointwise-exponential}, \citet{komlos1975approximation,komlos1976approximation} first show that for any $z > 0$,
\begin{equation}
  \Palphadottil \left ( \maxkn |\Lambdatil_k| > K_1(P) \log n + z \right ) \leq K_2(P) \exp \{ - K_3(P) z \}
\end{equation}
where $K_1(P)$, $K_2(P)$, and $K_3(P)$ all depend on $P$ in ways that are not explicit. \citet{sakhanenko1984rate} painted a more complete picture of what distributional features determine how $\maxkn |\Lambdatil_k|$ concentrates around 0 by showing that there exists a construction so that for any $z > 0$,
  \begin{equation}
    \Palphadottil \left ( \maxkn |\Lambdatil_k| \geq z \right ) \leq \left ( 1 + \lambda  \sqrt{n \Var_P(X)} \right ) \cdot \exp \left \{ -c \lambda z \right \},
  \end{equation}
  where $c$ is a universal constant and $\lambda > 0$ is any constant satisfying the inequality
  \begin{equation}\label{eq:sakhanenko-condition}
    \lambda \EE_P \left ( |X|^3 \exp \left \{ \lambda |X| \right \} \right ) \leq \Var_P(X).
  \end{equation}
  Notice that by H\"older's inequality, one can always find some $\lambda > 0$ satisfying the above when $P$ has a finite exponential moment. \citet{lifshits2000lecture} refers to the largest value of $\lambda$ satisfying \eqref{eq:sakhanenko-condition} as the \emph{Sakhanenko parameter} of $X$ under $P$. We will prove that uniformly \emph{lower-bounding} the Sakhanenko parameter introduces a degree of regularity for which uniform KMT approximation is possible (and to which it is in fact equivalent). We introduce this notion of regularity formally now.

\begin{definition}[Sakhanenko regularity]
  \label{definition:sakhanenko-regularity}
  Given a mean-zero random variable $X$ with a finite exponential moment on a probability space $\probspace$, we 
  say that $X$ has a \uline{Sakhanenko parameter} $\lambda(P)$ given by
\begin{equation}
  \lambda(P) := \sup \{ \lambda \geq 0 : \lambda \EE_P \left ( |X|^3 \exp \left \{ \lambda |X| \right \} \right ) \leq \Var_P (X) \}.
\end{equation}
Furthermore, given an index set $\Acal$ and an associated collection of probability spaces $\probspaces$, we say that $X$ is \uline{($\ubar\lambda$, $\Acal$)-Sakhanenko regular} if $X$ has a Sakhanenko parameter $\Acal$-uniformly lower-bounded by some positive $\uline \lambda$:
\begin{equation}\label{eq:sakhanenko-regularity}
  \inf_\alphain \lambda(\Palpha) \geq \uline \lambda > 0.
\end{equation}
\end{definition}
Sakhanenko regularity of a random variable $X$ implicitly imposes a uniform upper bound on its variance as seen from the following series of inequalities:
\begin{equation}\label{eq:sakhanenko-parameter-variance-inequalities}
  \left (\EE_\Palpha (X^2) \right )^{3/2} \leq \EE_\Palpha( |X|^3 ) \leq  \EE_\Palpha (|X|^3 \exp \{ \ubar \lambda |X| \} ) \leq \ubar \lambda^{-1} \EE_\Palpha (X^2),
\end{equation}
and hence $\Var_\Palpha(X) \equiv \EE_\Palpha(X^2) \leq \ubar \lambda^{-2}$; see also \citep[Eq. (3.1)]{lifshits2000lecture}.
Let us now examine some familiar Sakhanenko regular classes of distributions.
For a fixed $\sigma > 0$, let $\Acal_\sigma$ be an index set such that $X$ is mean-zero $\sigma$-sub-Gaussian for each $\alpha \in \Acal_\sigma$ --- i.e.~satisfying $\sup_{\alpha \in \Acal_\sigma}\EE_{\Palpha}  \exp \{ tX  \} \leq \exp \{ t^2 \sigma^2 / 2 \}$ for all $t \in \RR$ --- and such that the variance of $X$ is uniformly lower-bounded by $\ubar \sigma^2 > 0$ (for some $\ubar \sigma^2 \leq \sigma^2$). Then $X$ is $(\uline \lambda_\sigma, \Acal_\sigma)$-Sakhanenko regular where
\begin{equation}\label{eq:sub-gaussian-sakhanenko-regular}
    \ubar \lambda_\sigma := \frac{\ubar \sigma^2}{8e \sqrt{3}\sigma^3}.
\end{equation}
The above can be deduced from H\"older's inequality combined with the fact that $\sigma$-sub-Gaussian random variables have $q^\tth$ absolute moments uniformly upper-bounded by $\supalpha \EE_\Palpha|X|^q \leq q 2^{q/2} \sigma^q \Gamma(q/2)$; see \cref{lemma:sub-gaussian-are-sakhanenko-regular} in Appendix~\ref{section:auxiliary lemmas - sakhanenko regularity from subgaussianity}.

By virtue of being uniformly $\frac{1}{2} (b-a)$-sub-Gaussian, if $X$ has a uniformly bounded support $[a,b]$ for some $a \leq b$ on probability spaces indexed by $\Acal[a,b]$, then $X$ is $(\lambda[a,b], \Acal[a,b])$-Sakhanenko regular where $\lambda[a,b]$ is given by \eqref{eq:sub-gaussian-sakhanenko-regular} with $\sigma = \frac{1}{2}(b-a)$. While uniform sub-Gaussianity is a sufficient condition for Sakhanenko regularity, we will show in \cref{proposition:sakhanenko-param-bernstein-uniform-integrability} that it is not necessary and that $X$ is uniformly Sakhanenko regular if and only if it is uniformly sub-\emph{exponential} in a certain sense. To make this formal, recall the Bernstein parameter of $X$ under $P$:
\begin{equation}\label{eq:bernstein-parameter}
  b(P) := \inf \left \{ b > 0 : \EE_P| X |^q \leq \frac{q!}{2}  b^{q-2} \Var_P(X) ~~\text{for all}~~q = 3, 4, \dots\right \}.
\end{equation}
With \cref{definition:sakhanenko-regularity} and \cref{eq:bernstein-parameter} in mind, we are ready to state the equivalent characterizations of Sakhanenko regularity which will in turn be central to \cref{theorem:uniform-kmt-exponential}.

\begin{proposition}[Equivalent characterizations of Sakhanenko regularity]\label{proposition:sakhanenko-param-bernstein-uniform-integrability}
  Let $X$ be a random variable with mean zero on the collection of probability spaces $\probspaces$ with uniformly positive variance: $\inf_\alphain \Var_\Palpha(X) \geq \ubar \sigma^2 > 0$. The following four conditions are equivalent:
  \begin{enumerate}[label = (\roman*)]
  \item $X$ has an $\Acal$-uniformly bounded exponential moment, i.e.~there exists some $t > 0$ and $C>1$ so that
    \begin{equation}
      \supalpha \EE_\Palpha \left ( \exp \left \{ t |X| \right \} \right ) \leq C.
    \end{equation}
  \item $X$ has an $\Acal$-uniformly integrable exponential moment, i.e.~there exists some $t^\star > 0$ so that
  \begin{equation}\label{eq:uniform-integrability-exponential-moment}
      \lim_\Kto \supalpha \EE_\Palpha \left ( \exp \left \{ t^\star |X| \right \} \cdot \1 \left \{ \exp \left \{ t^\star|X| \right \} \geq K \right \}  \right ) = 0.
  \end{equation}
\item $X$ has an $\Acal$-uniformly upper-bounded Bernstein constant $\widebar b > 0$:
  \begin{equation}
    \supalpha b(\Palpha) \leq \widebar b.
  \end{equation}
  \item $X$ is $(\ubar \lambda, \Acal)$-Sakhanenko regular for some $\ubar \lambda > 0$.
  \end{enumerate}
  Furthermore, we have the following relations between the constants above:
  \begin{enumerate}[label = (\alph*)]
  \item If $(i)$ holds with $(t, C)$ then $(ii)$ holds with any $t^\star \in (0, t)$.
  \item If $(i)$ holds with $(t, C)$ then $(iii)$ holds with $\widebar b = 2 C \ubar \sigma \min \{ t\ubar \sigma , 1\}^{-3}$.
  \item If $(iii)$ holds with $\widebar b$, then $(iv)$ holds with $\ubar \lambda = (7\widebar b)^{-1}$.
  \item If $(iv)$ holds with $\ubar \lambda$, then $(iii)$ holds with $\widebar b = \ubar \lambda^{-1}$.
    \item If $(iv)$ holds with $\ubar \lambda$, then $(i)$ holds with $t = \ubar \lambda$ and $C = \ubar \lambda^{-3} + \exp \left \{ \ubar \lambda \right \}$.
  \end{enumerate}
\end{proposition}
A self-contained proof of \cref{proposition:sakhanenko-param-bernstein-uniform-integrability} is provided in \cref{proof:sakhanenko-regularity-equivalences}.
Note that when $\Acal = \{\alpha\}$ is a singleton and $X$ has a finite exponential moment, conditions $(i)$--$(iv)$ are always satisfied. As such, the consideration of Sakhanenko regular classes of distributions imposes no additional assumptions over the setting considered by \citet{komlos1975approximation,komlos1976approximation}; instead, it characterizes a class over which uniform strong approximations can hold, as will be seen in the following theorem. 

\begin{theorem}[Distribution-uniform Koml\'os-Major-Tusn\'ady approximation]\label{theorem:uniform-kmt-exponential}
Let $\infseqn{X_n}$ be \iid{} random variables with mean zero on $\probspaces$ and with a uniformly bounded and nondegenerate variance: {$0 < \inf_\alphain \Var_\Palpha(X) \leq \sup_\alphain \Var_\Palpha(X) < \infty$}. There exists a universal constant $c_{\ref*{theorem:uniform-kmt-exponential}} > 0$ such that $X$ is $(\ubar \lambda, \Acal)$-uniformly Sakhanenko regular for some $\ubar \lambda > 0$ if and only if there exist constructions $\probspacestilde$ satisfying
  \begin{equation}\label{eq:kmt-exponential-convergence}
    \lim_\mto \supalpha \Palphatil \left ( \supkm \left \lvert \frac{\Lambdatil_k}{\log k} \right \rvert > \frac{c_{\ref*{theorem:uniform-kmt-exponential}}}{\ubar \lambda} \right ) = 0.
  \end{equation}
\end{theorem}
Notice that \cref{theorem:uniform-kmt-exponential} implies the KMT approximation under finite exponential moments displayed in \eqref{eq:kmt-pointwise-exponential} since for a fixed probability space $\probspace$ and associated construction $\probspacetilde$,
\begin{equation}
  0 = \lim_{m \to \infty}\widetilde P \left ( \supkm \left \lvert \frac{\widetilde \Lambda_k}{\log k} \right \rvert > \frac{c_{\ref*{theorem:uniform-kmt-exponential}}}{\ubar \lambda} \right ) \geq \widetilde P \left ( \limsup_{n \to \infty}\left \lvert \frac{\widetilde \Lambda_n}{\log n} \right \rvert >\frac{ c_{\ref*{theorem:uniform-kmt-exponential}}}{\ubar \lambda} \right )
\end{equation}
by monotonicity.
Showing that Sakhanenko regularity is sufficient for the convergence in \eqref{eq:kmt-exponential-convergence} relies on a more refined nonasymptotic concentration inequality for partial sums of \iid{} random variables with finite exponential moments which we present now. 
\begin{theorem}[Nonasymptotic Koml\'os-Major-Tusn\'ady approximation with finite exponential moments]\label{theorem:concentration-kmt}
  Suppose $\infseqn{X_n}$ is an infinite sequence of mean-zero \iid{} random variables on $(\Omega, \Fcal, P)$ with variance $\sigma^2$ and Sakhanenko parameter $\widebar \lambda \equiv \widebar \lambda(P)$.
  Then, there exists a construction $\probspacetilde$ with the property that for any  $z > 0$ and any integer $m \geq 5$,
  \begin{equation}
    \Ptil \left ( \supkm \left \lvert \frac{\Lambdatil_k}{\log k} \right \rvert  > 4z \right ) \leq 2\sum_{n=n_m}^\infty \left (1 + \sqrt{2^{2^{n}}} \right ) \left (2^{  2^n } \right )^{-c_{\ref*{theorem:concentration-kmt}}\widebar \lambda z},
  \end{equation}
  where $n_m$ is the largest integer such that $\sum_{k=1}^{n_m - 1} 2^{2^k} + 1 \leq m$ and $c_{\ref*{theorem:concentration-kmt}} > 0$ is a universal constant.
\end{theorem}
The proof of \cref{theorem:concentration-kmt} found in \cref{proof:concentration-kmt} relies on extending a coupling inequality due to \citet{sakhanenko1984rate} for finite collections of random variables (see also \citep[Section 4.3]{einmahl1989extensions} and \citep[Theorem A]{shao1995strong}) to a common probability space for infinite collections and applying it on doubly exponentially-spaced epochs. The application of \cref{theorem:concentration-kmt} to establish sufficiency of Sakhanenko regularity in \cref{theorem:uniform-kmt-exponential} can be found in \cref{proof:uniform-kmt-exponential}. The proof of its necessity can be found in the same section and relies on a distribution-uniform analogue of the second Borel-Cantelli lemma \citep[Lemma 2]{waudby2024distribution}.

The following section considers strong approximations for partial sums of random variables with finite power moments.

\section{Random variables with finite power moments}\label{section:power-moments}

Consider a sequence of \iid{} random variables $\infseqn{X_n}$ on a space $\probspace$ having a finite power moment for some $q > 2$:
\begin{equation}\label{eq:power-moment-finite-pointwise}
  \EE_P |X|^q < \infty.
\end{equation}
Recall in \citet[Theorem 2]{komlos1976approximation} and \citet{major1976approximation} that under the condition in \eqref{eq:power-moment-finite-pointwise}, there is a construction $\probspacetilde$ so that
\begin{equation}\label{eq:kmt-power-moments-pointwise}
  \sum_{i=1}^n (\widetilde X_i - \widetilde Y_i) = o(n^{1/q})\quad \text{$\widetilde P$-almost surely.}\footnote{Throughout the remainder of the section, we refrain from using the shorthand $\Lambdatil_n$ that was introduced in \cref{section:exponential-moments} since we will provide other strong approximations that involve modified versions of $\infseqn{\Xtil_n}$ and $\infseqn{\Ytil_n}$ and we will make these modifications explicit through increased notational verbosity.}
\end{equation}
Similarly to the discussion in \cref{section:exponential-moments}, one of the challenges in arriving at a distribution-uniform analogue of \eqref{eq:kmt-power-moments-pointwise} is the presence of implicit distribution-dependent constants in the $\widetilde P$-almost sure $o(n^{1/q})$ asymptotic behavior. Moreover, it is not even clear \emph{a priori} what the right notion of a distribution-uniform strong Gaussian approximation is. With the goal of articulating that notion in mind, notice that \eqref{eq:kmt-power-moments-pointwise} is simply a statement about a particular weighted partial sum $n^{-1/q}\sum_{i=1}^n (\Xtil_i - \Ytil_i)$ \emph{vanishing} almost surely. As such, we will define uniform strong approximations in terms of analogous partial sums that vanish almost surely and \emph{uniformly} in a collection of (constructed) probability spaces, the essential ideas behind which have appeared in some prior work on distribution-uniform strong laws of large numbers \citep{chung_strong_1951,waudby2024distribution}. As in \citet{chung_strong_1951} (see also \citep[Definition 1]{waudby2024distribution}), a sequence $\infseqn{Z_n}$ is said to vanish almost surely and uniformly in a collection of probability spaces $\probspaces$ if
\begin{equation}\label{eq:uniform-convergence-almost-surely}
  \forall \eps > 0,\quad \lim_\mto \supalpha P_\alpha \left ( \supkm |Z_k| \geq \eps \right ) =0.
\end{equation}
Notice that when $\Acal = \{ \alphadot \}$ is a singleton, \eqref{eq:uniform-convergence-almost-surely} reduces to the statement that $Z_n \to 0$ with $P_\alphadot$-probability one since
\begin{equation}
  \lim_\mto P_\alphadot \left ( \supkm |Z_k| \geq \eps \right ) = 0 ~ ~ ~ \forall \eps > 0  \quad \text{if and only if} \quad Z_n \to 0~~P_\alphadot\text{-almost surely}.
\end{equation}
With the discussions surrounding \eqref{eq:kmt-power-moments-pointwise} and \eqref{eq:uniform-convergence-almost-surely} in mind, we now define what it means for strong approximations to hold distribution-uniformly.
\begin{definition}[Distribution-uniform strongly approximated processes]\label{definition:distribution-uniform-strong-coupling}
  For each $\alpha \in \Acal$, let $\infseqn{W_n^\brackalpha}$ be a stochastic process on $\probspacealpha$. Suppose that there exists a collection of constructions $\probspacestilde$ so that for each $\alpha \in \Acal$, the processes $\infseqn{\widetilde W_n^{(\alpha)}}$ and $\infseqn{\widetilde V_n^{(\alpha)}}$ are defined on $\probspacealphatil$ and the law of $\infseqn{\widetilde W_n^\brackalpha}$ under $\Ptil_\alpha$ is the same as that of $\infseqn{W_n^\brackalpha}$ under $P_\alpha$. Then we say that $( \infseqn{W_n^\brackalpha} )_{\alpha \in \Acal}$ is $\Acal$-uniformly strongly approximated by $( \infseqn{\widetilde V_n^\brackalpha} )_{\alphain}$ if $\widetilde W_n^\brackalpha - \widetilde V_n^\brackalpha$ vanishes $\Acal$-uniformly almost surely in the sense of \eqref{eq:uniform-convergence-almost-surely}.
  As a shorthand, 
  we write
  \begin{equation}
    \widetilde W_n - \widetilde V_n = \oAcalas(1).
  \end{equation}
  We say that $\widetilde W_n - \widetilde V_n = \oAcalas(r_n)$ for some monotone positive sequence $\infseqn{r_n}$ if $r_n^{-1} (\widetilde W_n - \widetilde V_n) = \oAcalas(1)$.
\end{definition}
We now have the requisite definitions to state the main result of this section which serves as a uniform generalization of the KMT approximation for power moments and effectively extends certain results about distribution-uniform strong laws of large numbers \citep{chung_strong_1951,waudby2024distribution} to moments larger than 2.
\begin{theorem}[Distribution-uniform Koml\'os-Major-Tusn\'ady approximation for finite power moments]\label{theorem:uniform-kmt-power-moments}
  Let $\infseqn{X_n}$ be \iid{} random variables with mean zero on the collection of probability spaces $\probspaces$. Let $X$ have a uniformly bounded and nondegenerate variance, meaning that there exist $0 < \ubar \sigma^2 \leq \widebar \sigma^2 < \infty$ so that $\ubar \sigma^2 < \Var_\Palpha (X) < \widebar \sigma^2$ for all $\alphain$. Then $X$ has an $\Acal$-uniformly integrable $q^\tth$ moment for $q > 2$,
  \begin{equation}
    \lim_\Kto \supalpha \EE_\Palpha \left ( |X|^q \1 \{ |X|^q \geq K \}  \right ) = 0
  \end{equation}
  if and only if there exist constructions $\probspacestilde$ so that
  \begin{equation}
    \sum_{i=1}^n \widetilde X_i - \sum_{i=1}^n \widetilde Y_i = \oAcalas(n^{1/q}).
  \end{equation}
\end{theorem}
The distribution-pointwise results of \citet{komlos1975approximation,komlos1976approximation} and \citet{major1976approximation} for finite power moments as in \eqref{eq:kmt-power-moments-pointwise} can be immediately derived from \cref{theorem:uniform-kmt-power-moments} by taking $\Acal = \{ \alpha \}$. Similar to the case of exponential moments discussed in \cref{section:exponential-moments}, the necessity of uniform integrability follows from a uniform generalization of the second Borel-Cantelli lemma \citep[Lemma 2]{waudby2024distribution} while sufficiency follows from first deriving a nonasymptotic concentration inequality with constants that are either universal or depend on the distribution in explicit ways. As will become apparent shortly, however, the use of this inequality in the proof of \cref{theorem:uniform-kmt-power-moments} (found in \cref{proof:uniform-kmt-power-moments}) is more delicate and less direct than the use of \cref{theorem:concentration-kmt} in the proof of \cref{theorem:uniform-kmt-exponential} for finite exponential moments. Nevertheless, we present this inequality here and later discuss how it can be used to prove the sufficiency half of \cref{theorem:uniform-kmt-power-moments}. Throughout, we will write $a_n \nearrow \infty$ to mean that a real sequence $\infseqn{a_n}$ is positive, nondecreasing, and diverging.

\begin{theorem}[Nonasymptotic strong approximation with finite power moments]\label{theorem:concentration-power-moments}
  Let $\infseqn{X_n}$ be independent, mean-zero random variables on the probability space $(\Omega, \Fcal, P)$. Suppose that for some $q > 2$, we have $\EE_P|X_k|^q < \infty$ for each $k \in \NN$ and that the random variables $\infseqn{X_n}$ are eventually nondegenerate, i.e.~$\liminf_{n \to \infty} \Var_P(X_n) > 0$.
Let $\infseqn{\ubar a_n}$ be any positive sequence such that $\ubar a_n \nearrow \infty$ and
\begin{equation}
  \sum_{k=1}^\infty \frac{\EE_P |X_k|^q}{\ubar a_k^q} < \infty,
\end{equation}
and let $\infseqn{a_n}$ be any other nondecreasing sequence so that $\ubar a_n \leq a_n$ for each $n$ and $\ubar a_n / a_n \to 0$.
Then there exists a construction $\probspacetilde$ with $\infseqn{\Ytil_n}$ being marginally independent Gaussian random variables where $\EE_\Ptil (\Ytil_k) = \EE_P(X_k)$ and $\Var_\Ptil(\Ytil_k) = \Var_P(X_k)$ for each $k \in \NN$ so that for any $\eps > 0$ and any $m \geq 1$,
\begin{align}
  \Ptil \left ( \supkm \frac{|\sum_{i=1}^k (\Xtil_i - \Ytil_i)|}{a_k} \geq \eps \right ) \leq  \frac{C_{\ref*{theorem:concentration-power-moments}}(q)}{\eps^q} \left ( \sum_{k=m}^\infty \frac{\EE_P |X_k|^q}{\ubar a_k^q} + \frac{\ubar a_{n_m}^q}{a_{n_m}^q}\sum_{k=1}^\infty \frac{\EE_P|X_k|^q}{\ubar a_k^q} \right ),
\end{align}
where $C_{ \ref*{theorem:concentration-power-moments}}(q)$ is a constant that depends only on $q$ and where
\begin{equation}
    n_m = \min \left \{ n \in \NN : \log_2 \left ( \frac{\sum_{k=1}^\infty \EE_P|X_k|^q / \ubar a_k^q}{ \sum_{k=n}^\infty \EE_P|X_k|^q / \ubar a_k^q } \right ) \geq \left \lfloor \log_2 \left ( \frac{\sum_{k=1}^\infty \EE_P|X_k|^q / \ubar a_k^q}{ \sum_{k=m}^\infty \EE_P|X_k|^q / \ubar a_k^q } \right ) \right \rfloor
    \right \}.
\end{equation}
\end{theorem}
The proof of \cref{theorem:concentration-power-moments} in \cref{proof:concentration-power-moments} relies on a polynomial coupling inequality due to \citet{sakhanenko1985estimates,sakhanenko2006estimates} for finite collections of random variables. Note that it is straightforward to use \cref{theorem:concentration-power-moments} to deduce the asymptotic and distribution-pointwise strong approximation of \citet[Theorem 1.3]{shao1995strong} which states that if there exists a sequence $a_n \nearrow \infty$ for which $\sum_{k=1}^\infty \EE_P |X_k|^q / a_k^q < \infty$ and if $\liminf_n \Var_P(X_n) > 0$, then there exists a construction so that $\sum_{i=1}^n (\Xtil_i - \Ytil_i) = o(a_n)$ with probability one. Indeed, if $\sum_{k=1}^\infty \EE_P|X_k|^q / a_k^q < \infty$, then there exists a sequence $\ubar a_n \nearrow \infty$ so that $\ubar a_n \leq a_n$ and $\ubar a_n / a_n \to 0$ and yet
\begin{equation}
  \sum_{k=1}^\infty \frac{\EE_P |X_k|^q}{\ubar a_k^q} < \infty.
\end{equation}
Applying \cref{theorem:concentration-power-moments} and observing that $n_m \to \infty$ as $m \to \infty$, we have that there is a construction $\probspacetilde$ so that for any $\eps > 0$,
\begin{equation}
  \lim_\mto \Ptil \left ( \supkm \frac{|\sum_{i=1}^k (\Xtil_i - \Ytil_i)|}{a_k}  \geq \eps \right ) = 0,
\end{equation}
or equivalently, $\sum_{i=1}^n (\Xtil_i - \Ytil_i) = o(a_n)$ with $\Ptil$-probability one. In fact, the above follows immediately from the following asymptotic result which serves as a distribution-uniform generalization of \citet[Theorem 1.3]{shao1995strong}.
\begin{corollary}\label{corollary:shao-strong-approx-uniform}
  Let $\infseqn{X_n}$ be independent, mean-zero random variables on the collection of probability spaces $\probspaces$ satisfying the following two uniform boundedness and integrability conditions for some $q > 2$ and some sequence $a_n \nearrow \infty$:
  \begin{equation}\label{eq:shao-uniform-approx-conditions}
    \sup_\alphain \sum_{k=1}^\infty \frac{\EE_\Palpha|X_k|^q}{a_k^q} < \infty\quad\text{and}\quad \lim_\mto \sup_\alphain \sum_{k=m}^\infty \frac{\EE_\Palpha|X_k|^q}{a_k^q} = 0.
  \end{equation}
  Suppose that in addition, the variances of $\infseqn{X_n}$ are uniformly positive in the limit:
  \begin{equation}\label{eq:lower-bounded-liminf-variance}
    \liminf_\nto \inf_\alphain \Var_\Palpha(X_n) > 0.
  \end{equation}
  Then there exist constructions with marginally independent mean-zero Gaussian random variables $\infseqn{\widetilde Y_n}$ with $\Var_\Palphatil(\Ytil_n) = \Var_\Palpha(X_n)$ for each $n \in \NN$ and $\alphain$ so that
  \begin{equation}
    \sum_{k=1}^n (\widetilde X_k - \widetilde Y_k) = \oAcalas\left (a_n \right ).
  \end{equation}
\end{corollary}
A proof of \cref{corollary:shao-strong-approx-uniform} can be found in \cref{proof:shao-strong-approx-uniform}. In addition to being a distribution-generalization of \citet[Theorem 1.3]{shao1995strong}, \cref{corollary:shao-strong-approx-uniform} serves as a Gaussian approximation analogue of the uniform strong law of large numbers for independent random variables in \citep[Theorem 2]{waudby2024distribution}.
While it may be easy to see how \cref{corollary:shao-strong-approx-uniform} follows from the concentration inequality in \cref{theorem:concentration-power-moments} when instantiated with the same value of $q > 2$, the same cannot be said for \cref{theorem:uniform-kmt-power-moments} since obtaining the approximation rate of $o(n^{1/q})$ from a naive application of \cref{theorem:concentration-power-moments} would fail due to $\sum_{k=1}^\infty \EE_P|X_k|^q / k$ not being summable. Indeed, the proof of \cref{theorem:uniform-kmt-power-moments} crucially relies on \cref{corollary:shao-strong-approx-uniform} in an intermediate step but applied to truncated random variables $X_k \1 \{ |X_k| \leq k^{1/q} \}; k \in \NN$ and doing so with a higher moment $p > q$ and $a_k = k^{1/q}$ so that the rate of approximation remains $\oAcalas(n^{1/q})$. The error introduced from upper-truncating $X_k$ at the level $k^{1/q}$ is controlled by appealing to a certain stochastic and uniform generalization of Kronecker's lemma \citep[Lemma 1]{waudby2024distribution}, the application of which centrally exploits uniform integrability of the $q^\tth$ moment.
The application of \cref{corollary:shao-strong-approx-uniform} to \cref{theorem:uniform-kmt-power-moments} relies on the additional structure resulting from random variables being identically distributed; see the proof of \cref{theorem:shao-einmahl-lifshits}. The ``gap'' in the rates of convergence between independent and \iid{} settings directly mirrors the relationship between the pointwise results of \citet{komlos1975approximation,komlos1976approximation} and \citet[Theorem 1.3]{shao1995strong} or the relationship between independent and \iid{} strong laws of large numbers (see for example \citep[Theorems 1(i) and 2]{waudby2024distribution}).

Let us now give a more detailed discussion of how \cref{corollary:shao-strong-approx-uniform} is used in the proof of \cref{theorem:uniform-kmt-power-moments}, leaving the formal details to \cref{proof:uniform-kmt-power-moments}. We decompose the sum $\sum_{k=1}^n X_k$ into three terms:
\begin{equation}\label{eq:truncated-decomposition}
  \sum_{k=1}^n X_k = \sum_{k=1}^n \left [ X_k^\leq - \EE_\Palpha (X_k^\leq) \right ] + \sum_{k=1}^n X_k^> - \sum_{k=1}^n \EE_\Palpha ( X_k^> ),
\end{equation}
where $X_k^\leq := X_k \1 \{ |X_k| \leq k^{1/q} \}$ and $X_k^> := X_k \1 \{ |X_k| > k^{1/q} \}$ for each $k \in \NN$. Noting that the first term contains sums of independent mean-zero random variables, \cref{corollary:shao-strong-approx-uniform} is applied to the first term but with a power of $p > q$ in place of $q$. However, this only results in a uniform strong approximation with partial sums of independent mean-zero Gaussian random variables $\infseqn{\Ytil_n^\brackleq}$ with variances given by $\Var_\Palpha(X_k^\leq)$ for each $k \in \NN$. Nevertheless, we show that the difference between $\sum_{k=1}^n \Ytil_k^\brackleq$ and $\sum_{k=1}^n \Ytil_k$ for an appropriately constructed mean-zero \iid{} sequence $\infseqn{\Ytil_n}$ with variances given by $\Var_\Palpha(X)$ is uniformly $\oAcalas(n^{1/q})$, and this is achieved via an application of a uniform Khintchine--Kolmogorov convergence theorem as well as a stochastic and uniform Kronecker lemma \citep[Theorem 3 \& Lemma 1]{waudby2024distribution}. The second and third terms in \eqref{eq:truncated-decomposition} are shown to vanish using some standard truncation arguments and the same uniform Kronecker lemma.

\section{Proofs from \cref*{section:exponential-moments}}\label{section:proofs-exponential-moments}

\subsection{Proof of \cref*{proposition:sakhanenko-param-bernstein-uniform-integrability}}\label{proof:sakhanenko-regularity-equivalences}

The proof of \cref{proposition:sakhanenko-param-bernstein-uniform-integrability} relies on the following lemma.

\begin{lemma}[Bounds on polynomial moments from exponential ones]\label{lemma:bound on polynomial moments from exponential}
  Let $Y$ be a random variable on $\probspace$ so that for some $t, C > 0$,
  \begin{equation}
    \EE_P \left ( \exp \left \{ t|Y| \right \} \right ) \leq C.
  \end{equation}
  Then for any integer $q \geq 2$, the $q^\tth$ moment of $Y$ can be upper bounded as
  \begin{equation}
    \EE_P |Y|^q \leq C t^{-q} q! ~ .
  \end{equation}
\end{lemma}

\begin{proof}
  By Markov's inequality, we first observe that
  \begin{align}
    P \left ( |Y| \geq z \right ) = P \left ( \exp \left \{ t |Y| \right \} \geq \exp \left \{ t z \right \} \right )
    \verbose{\leq \frac{\EE_P \left ( \exp \left \{ t |Y| \right \} \right )}{\exp \left \{ t z \right \}}\\}
      \leq C \exp \left \{ -t z \right \},
  \end{align}
  and hence for any integer $q \geq 2$, we write the expectation $\EE_P |Y|^q$ as an integral of tail probabilities so that
  \begin{align}
    \EE_P |Y|^q &= \int_0^\infty P ( |Y|^q \geq z ) dz 
    \verbose{
    = \int_0^\infty P ( |Y| \geq z^{1/q} ) dz\\
    }
    \leq \int_0^\infty C\exp \{ -t z^{1/q} \}  dz
    = C q \int_0^\infty \exp \{ -t u \}  u^{q-1} du,
  \end{align}
  where we used the change of variables $z = u^q$. Using another change of variables given by $u = w / t$, we continue from the above and notice that
  \begin{align}
    \EE_P |Y|^q &\leq C q \int_0^\infty  \exp \{-t u \}  u^{q-1} du
    = C t^{-q} q \int_0^\infty e^{-w}  w^{q-1} dw 
    = C t^{-q} q \Gamma(q) 
    = C t^{-q} q!,
  \end{align}
  where we used the definition of the gamma function $\Gamma(q)$ and the fact that $\Gamma(q) = (q-1)!$ when $q$ is a positive integer.
  This completes the proof.
\end{proof}

\begin{proof}[Proof of \cref{proposition:sakhanenko-param-bernstein-uniform-integrability}]
  The proof proceeds by showing that $(i) \implies (iii) \implies (iv) \implies (i)$ and later justifying the equivalence between $(i)$ and $(ii)$, all the while keeping track of constants to justify the relations described in $(a)$--$(e)$.
  \paragraph{Showing that $(i) \implies (iii)$} Let $C \geq 1$ and $t > 0$, and suppose that
  \begin{equation}
    \supalpha \EE_\Palpha \left ( \exp \left \{ t |X| \right \} \right ) \leq C.
  \end{equation}
  First consider $Y := X / \ubar \sigma$ where $\ubar \sigma^2$ is the $\Acal$-uniform lower-bound on the variance of $X$, and thus we of course have
  \begin{equation}\label{eq:i implies ii assumption Y}
    \supalpha \EE_\Palpha \left ( \exp \left \{ t' |Y| \right \} \right ) \leq C,
  \end{equation}
  where $t' := t\ubar \sigma$.
  We proceed by deriving a deviation inequality for $|Y|$, then use it within the integrated tail probability representation of the $q^\tth$ moment for any integer $q \geq 3$ to obtain an upper-bound on that moment, and ultimately re-writing the final expression in terms of the Bernstein parameter as in \eqref{eq:bernstein-parameter}.
  To this end, notice that the left-hand side of \eqref{eq:i implies ii assumption Y} is nondecreasing in $t' > 0$ and thus
  \begin{equation}
    \supalpha \EE_\Palpha \left ( \exp \left \{ t_\star |Y| \right \} \right ) \leq C,
  \end{equation}
  where $t_\star := \min \{ t', 1 \}$. Applying \cref{lemma:bound on polynomial moments from exponential}, we have that
  \begin{align}
    \EE_\Palpha |Y|^q &\leq C t_\star^{-q} q! 
                = \frac{1}{2} q! \cdot \frac{2Ct_\star^{-2}}{t_\star^{q-2}} 
                \leq \frac{1}{2} q! \cdot \frac{(2Ct_\star^{-2})^{q-2}}{t_\star^{q-2}} 
                = \frac{1}{2} q! \cdot \left ( 2Ct_\star^{-3} \right )^{q-2},
  \end{align}
  where, in the inequality above, we used the fact that $t_\star := \min\{1, t'\} \leq 1$ and that $C \geq 1$ so that $2Ct_\star^{-2} \leq (2C t_\star^{-2})^{q-2}$ for any integer $q \geq 3$. Returning to the original random variable $X$, notice that
  \begin{align}
    \EE_\Palpha |X|^q &= \EE_\Palpha | \ubar \sigma Y |^q 
                      \leq  \frac{1}{2} q! \cdot \left ( 2C \ubar \sigma t_\star^{-3} \right )^{q-2} \ubar \sigma^2 
                      \leq  \frac{1}{2} q! \cdot \left ( 2C \ubar \sigma t_\star^{-3} \right )^{q-2} \cdot \Var_\Palpha(X).
  \end{align}
  Therefore, $X$ has a Bernstein parameter $b(\Palpha)$ upper bounded by $2C \ubar \sigma t_\star^{-3}$ for each $\alphain$, and thus
  \begin{equation}
    \supalpha b(\Palpha) \leq 2C \ubar \sigma t_\star^{-3}.
  \end{equation}
  The upper-bound on the variance can be obtained from an application of \cref{lemma:bound on polynomial moments from exponential} with $q = 2$:
  \begin{equation}
    \sup_\alphain \Var_\Palpha(X) \leq 2 C t^{-2},
  \end{equation}
  which completes the proof of $(i) \implies (iii)$.
  \paragraph{Showing that $(iii) \implies (iv)$ and the inequality given in $(d)$}
  The following arguments are similar to those found in \citep[\S 3; pp. 7--8]{lifshits2000lecture} but we reproduce them here for the sake of completeness. Using the Taylor expansion of $x \mapsto e^x$ around 0 and appealing to Fubini's theorem, we have for any $\lambda  < \widebar b^{-1}$,
  \begin{align}
    \EE_\Palpha \left ( |X|^3 \exp \left \{ \lambda |X| \right \} \right ) &= \EE_\Palpha \left ( |X|^3 \sum_{k=0}^\infty  |\lambda X|^k / k! \right ) \\
    \verbose{
        &= \sum_{k=0}^\infty \frac{|\lambda|^k}{k!} \EE_\Palpha \left ( |X|^{3+k} \right )\\}
    &\leq \sum_{k=0}^\infty \frac{|\lambda|^k}{k!} \frac{(3+k)!}{2} \widebar b^{k+1} \Var_\Palpha (X)\\
    \verbose{
    &= \frac{\Var_\Palpha (X) \widebar b }{2} \sum_{k=0}^\infty (k+3)(k+2)(k+1) |\lambda \widebar b|^k\\
    &= \frac{\Var_\Palpha (X) \widebar b}{2} \frac{6}{(\lambda\widebar b - 1)^4} \\
    }
    &= \frac{3 \Var_\Palpha (X) \widebar b}{(\lambda\widebar b - 1)^4},
  \end{align}
  and so we can write for any $\lambda < \widebar b^{-1}$,
  \begin{equation}
    \lambda \EE_\Palpha \left ( |X|^3 \exp \left \{ \lambda |X| \right \} \right ) \leq \frac{3 \lambda \widebar b }{(\lambda\widebar b - 1)^4} \cdot \Var_\Palpha (X),
  \end{equation}
  and it is easy to check that the first factor on the right-hand side is always smaller than 1 whenever $\lambda \leq 1/(7 \widebar b)$. This completes the proof of $(iii) \implies (iv)$.

  Moving on to the relation between constants given in $(d)$,
  suppose that $X$ is ($\ubar \lambda$, $\Acal$)-Sakhanenko regular. Then notice that for each $\alphain$ and any $k \geq 3$,
  \begin{align}
    \frac{\ubar \lambda^{k-2}}{(k-3)!} \EE_\Palpha |X|^k &= \ubar \lambda \EE_\Palpha \left [ |X|^3 \frac{ (\ubar \lambda |X|)^{k-3} }{(k-3)!} \right ]\leq \ubar \lambda \EE_\Palpha \left [ |X|^3 \exp \left \{ \ubar \lambda |X| \right \} \right ] \leq \Var_\Palpha (X),
  \end{align}
  and hence we can upper bound the $k^\tth$ moment as
  \begin{align}
    \EE_\Palpha |X|^k \leq (k-3)! \cdot \ubar \lambda^{2-k} \Var_\Palpha(X) \leq \frac{k!}{2} \cdot (1/\ubar \lambda)^{k-2} \cdot \Var_\Palpha(X),
  \end{align}
  demonstrating that $X$ satisfies the Bernstein condition uniformly in $\Acal$ with parameter $1/\ubar \lambda$. This completes the proof of relation $(d)$.
  
  \paragraph{Showing that $(iv) \implies (i)$}
  Let $X$ be ($\ubar \lambda$, $\Acal$)-Sakhanenko regular so that
  \begin{equation}
    \ubar \lambda\EE_\Palpha \left ( |X |^3 \exp \left \{ \ubar \lambda |X | \right \} \right ) \leq \Var_\Palpha(X) \leq \ubar \lambda^{-2},
  \end{equation}
  for every $\alphain$.
  Using the above, and performing a direct calculation, we have that for any $\alphain$,
  \begin{align}
    \EE_\Palpha \left ( \exp \left \{ \ubar \lambda |X| \right \} \right ) &= \EE_\Palpha \left ( \exp \left \{ \ubar \lambda |X| \right \} \cdot \1 \{ |X|^3 > 1 \} \right ) + \EE_\Palpha \left ( \exp \left \{ \ubar \lambda |X| \right \} \cdot \1 \{ |X|^3 \leq 1 \} \right )\\
    &\leq \EE_\Palpha \left ( |X|^3 \exp \left \{ \ubar \lambda |X| \right \} \cdot \1 \{ |X|^3 > 1 \} \right ) + \EE_\Palpha \left ( \exp \left \{ \ubar \lambda \right \} \cdot \1 \{ |X|^3 \leq 1 \} \right )\\
    &\leq \EE_\Palpha \left ( |X|^3 \exp \left \{ \ubar \lambda |X| \right \} \right ) +  \exp \left \{ \ubar \lambda \right \}\\
    &\leq \ubar \lambda^{-3} +  \exp \left \{ \ubar \lambda \right \},
  \end{align}
  which completes the proof of $(iv) \implies (i)$.
  \paragraph{Showing that $(i)$ and $(ii)$  are equivalent}
  This follows from the condition in $(i)$ combined with the de la Vall\'ee Poussin criterion of uniform integrability. In particular, if $(i)$ holds with some $t> 0$ and $C \geq 1$, then $\exp \left \{  (t/2) |X| \right \}$ is uniformly integrable. In the other direction, if $\exp \left \{ t^\star |X| \right \}$ is uniformly integrable, then $\sup_\alphain \EE_\Palpha \exp \left \{ t^\star |X| \right \} < \infty$.
  This completes the proof of \cref{proposition:sakhanenko-param-bernstein-uniform-integrability}.
\end{proof}

\subsection{Proof of \cref*{theorem:concentration-kmt}}\label{proof:concentration-kmt}

We first prove \cref{theorem:concentration-kmt} and later apply it to the sufficiency half of \cref{theorem:uniform-kmt-exponential}.
The proof of \cref{theorem:concentration-kmt} relies on an exponential coupling inequality for finitely many random variables due to \citet{sakhanenko1985estimates}.

\begin{lemma}[Sakhanenko's exponential inequality]\label{lemma:sakhanenko-exponential}
  Let $(X_1, \dots, X_n)$ be mean-zero \iid{} random variables with variance $\sigma^2$ and with Sakhanenko parameter $\widebar \lambda > 0$ on the probability space $(\Omega, \Fcal, P)$. Then there exists a construction $(\widetilde \Omega, \widetilde \Fcal, \widetilde P)$ so that the tuple $(\widetilde X_1, \dots, \widetilde X_n)$ has the same distribution as $(X_1, \dots, X_n)$ and $(\widetilde Y_1, \dots, \widetilde Y_n)$ consists of \iid{} Gaussians with means zero and variances $\sigma^2$ and so that for any $0 < \lambda < \widebar \lambda$,
  \begin{equation}\label{eq:sakhanenko-exponential-inequality}
    \EE_{\widetilde P} \left ( \exp \left \{  c\lambda \max_{1 \leq k \leq n} \left \lvert \sum_{i=1}^k (\widetilde X_i - \widetilde Y_i) \right \rvert  \right \}\right ) \leq 1 + \lambda \sqrt{n} \sigma,
  \end{equation}
  where $c >0 $ is a universal constant.
\end{lemma}

\begin{proof}[Proof of \cref{theorem:concentration-kmt}]
  The proof proceeds in 5 steps. 
  The first two steps borrow inspiration from some arguments found in \citep[Corollary 2.3]{lifshits2000lecture} where we apply Sakhanenko's exponential inequality infinitely many times to $\infseqn{X_n}$ in batches of size $2^{2^n}$ for each integer $n \geq 1$. For each of these applications, we construct a product measurable space and sequences $\infseqn{\widetilde X_n}$ and $\infseqn{\widetilde Y_n}$ which agree with those in the couplings resulting from successive applications of Sakhanenko's inequality. In the third step, we partition the crossing event $\left \{ \supkm \left ( \left \lvert \sum_{i=1}^k (\widetilde X_i - \widetilde Y_i) \right  \rvert / \log k \right ) > 4z \right \} $ for any $z > 0$ and $m \geq 5$ into pieces of size $2^{2^n}$. Steps~4 and~5 focus on bounding terms from the individual crossing events from Step~3 with high probability. 

  Throughout, denote the demarcation sizes by $d(n) := 2^{2^n}$ and their partial sums as $\Dcal(n) := \sum_{k=1}^n d(k)$ for any $n \in \{1, 2, \dots\}$ and let $\Dcal(0) = 0$. Furthermore, for any sequences of random variables $\infseqn{A_n}$ and $\infseqn{B_n}$, let $\Lambda_n(A, B)$ denote the difference of their partial sums:
  \begin{equation}\label{eq:discrepancy}
    \Lambda_n(A, B) := \sum_{i=1}^n (A_i - B_i).
  \end{equation}

  \paragraph{Step 1: Applying Sakhanenko's inequality at doubly exponential demarcations}
  Fix $0 < \lambda < \widebar \lambda$. For every $n \in \{1, 2, \dots\}$, apply Sakhanenko's exponential inequality (\cref{lemma:sakhanenko-exponential}) to the tuple $(X_{\Dcal(n-1) + 1}, \dots, X_{\Dcal(n)})$ of size $d(n)$ to obtain the probability space $(\widetilde \Omega_n, \widetilde \Fcal_n, \widetilde P_n)$ and the random variables $(\widetilde X_{\Dcal(n-1)+1}^\brackn, \dots, \widetilde X_{\Dcal(n)}^\brackn)$ so that
  \begin{equation}
    \EE_{\Ptil_n} \left [ \exp \left \{ c\lambda \max_{\Dcal(n-1) < k \leq \Dcal(n)} \left \lvert \sum_{i=\Dcal(n-1)+1}^k (\Xtil_i^\brackn - \Ytil_i^\brackn) \right \rvert \right \} \right ] \leq 1 + \lambda \sqrt{d(n)} \sigma,
  \end{equation}
  and hence
  for any $z > 0$,
  \begin{equation}\label{eq:exponential-proof-sakhanenko}
    \Ptil_n \left ( \max_{\Dcal(n-1)< k \leq \Dcal(n) } \left \lvert \sum_{i=\Dcal(n-1)+1}^{k} (\widetilde X_i^\brackn - \widetilde Y_i^\brackn) \right \rvert > z \right ) \leq \left ( 1 + \lambda \sqrt{d(n)} \sigma \right ) \cdot \exp \left \{ -c \lambda z \right \}.
  \end{equation}

  \paragraph{Step 2: Constructing the couplings on a single probability space}
  Using the probability spaces from Step 1, construct a new probability space $(\Omega^\star, \Fcal^\star, P^\star)$ by taking the product $\Omega^\star := \prod_{m=1}^\infty \widetilde \Omega_m$ where the resulting $\Fcal^\star$ and $P^\star$ exist by Kolmogorov's extension theorem. On this new space, we construct random variables $\infseqn{\widetilde X_n}$ and $\infseqn{\widetilde Y_n}$ defined for each $\omega \equiv (\omega_1, \omega_2, \dots) \in \Omega$ as
  \begin{align}
    \widetilde X_i(\omega) &:= \widetilde X_{i}^\brackn \equiv \widetilde X_{i}^\brackn(\omega_{\Dcal(n-1)+1}, \dots, \omega_{\Dcal(n)})\quad\text{and} \label{eq:exponential-proof-Xtil-def}\\
    \widetilde Y_i(\omega) &:= \widetilde Y_{i}^\brackn \equiv \widetilde Y_{i}^\brackn(\omega_{\Dcal(n-1)+1}, \dots, \omega_{\Dcal(n)}) \quad\text{whenever } \Dcal(n-1) + 1 \leq i \leq \Dcal(n).\label{eq:exponential-proof-Ytil-def}
  \end{align}
  In other words, the sequence $\infseqn{\widetilde X_n}$ agrees with the values of $\widetilde X_1^{(1)}, \dots, \widetilde X_{\Dcal(1)}^{(1)}$, then $\widetilde X_{\Dcal(1)+1}^{(2)}, \dots, \widetilde X_{\Dcal(2)}^{(2)}$, and so on.

  \paragraph{Step 3: Partitioning $\NN$ into doubly exponentially spaced epochs}
  Recalling the definition of $\Lambda_n$ from \eqref{eq:discrepancy} and denoting $\widetilde \Lambda_n := \Lambda_n(\Xtil, \Ytil)$, note that for any $n \in \{1,2,\dots\}$, we can decompose its maximum over any interval $[a, b];\ a < b$ where $a,b$ are both positive integers as
  \begin{equation}\label{eq:exponential-proof-max-decomposition}
    \max_{a < k \leq b} \widetilde \Lambda_k  = \max_{a < k \leq b} \{\widetilde  \Lambda_k - \widetilde \Lambda_{a} \} + \widetilde \Lambda_{a}.
  \end{equation}
  We will now break ``time'' (i.e.~the natural numbers $\NN$) up into epochs of size $d(n) := 2^{2^n}$ and bound the first and second terms in the right-hand side of \eqref{eq:exponential-proof-max-decomposition} with high probability.
  Letting $n_m$ be the largest integer so that $\Dcal(n_m - 1) + 1 \leq m$, we have that for any $m \geq 5$,
  \begin{align}
    &\Pstar \left ( \supkm \left \lvert \frac{\Lambdatil_k}{\log k} \right \rvert > 4z \right ) 
    \leq \Pstar \left ( \exists n \in \{n_m, n_m + 1, \dots\} : \max_{\Dcal(n-1) < k \leq \Dcal(n)} \left \lvert \frac{\Lambdatil_k}{\log k} \right \rvert > 4z \right ).
  \end{align}
  Union bounding over $n_m, n_m+1, \dots$, lower bounding $\log k$ by $\log \Dcal(n-1)$ whenever $k > \Dcal(n-1)$, and applying the triangle inequality, we have that 
  \begin{align}
    &\Pstar \left ( \supkm \left \lvert \frac{\Lambdatil_k}{\log k} \right \rvert > 4z \right )\\
    \leq\ & \sum_{n=n_m}^\infty \Pstar \left ( \max_{\Dcal(n-1) < k \leq \Dcal(n)} \left\lvert \frac{\Lambdatil_k}{\log k} \right \rvert > 4z \right ) \\
    \verbose{
    \leq\ & \sum_{n=n_m}^\infty \Pstar \left ( \max_{\Dcal(n-1) < k \leq \Dcal(n)} \left \lvert \Lambdatil_k \right \rvert > 4z\log \Dcal(n-1) \right ) \\
    \leq\ & \sum_{n=n_m}^\infty \Pstar \left ( \max_{\Dcal(n-1) < k \leq \Dcal(n)} \left \lvert \Lambdatil_k - \Lambdatil_{\Dcal(n-1)} \right \rvert + \left \lvert\Lambdatil_{\Dcal(n-1)} \right \rvert > 4z\log \Dcal(n-1) \right ) \\
    }
    \leq\ & \underbrace{\sum_{n=n_m}^\infty \Pstar \left ( \max_{\Dcal(n-1) < k \leq \Dcal(n)} \left \lvert \Lambdatil_k - \Lambdatil_{\Dcal(n-1)} \right \rvert > 2z\log \Dcal(n-1) \right )}_{(\star)} +\\
    \quad &\underbrace{\sum_{n=n_m}^\infty \Pstar \left ( \left \lvert \Lambdatil_{\Dcal(n-1)} \right \rvert >  2z \log \Dcal(n-1) \right )}_{(\dagger)}.
\end{align}
 We will now focus on bounding $(\star)$ and $(\dagger)$ with high $P^\star$-probability separately.

  \paragraph{Step 4: Obtaining the desired high-probability bound on $(\star)$}
  Writing out the summands inside $(\star)$ from the previous step for any $n \geq n_m$ and recalling how $\Xtil$ and $\Ytil$ were defined in \eqref{eq:exponential-proof-Xtil-def} and \eqref{eq:exponential-proof-Ytil-def}, respectively, observe that
  \begin{align}
    &\Pstar \left ( \max_{\Dcal(n-1) < k \leq \Dcal(n)} \lvert \Lambdatil_k - \Lambdatil_{\Dcal(n-1)} \rvert > 2z \log \Dcal(n-1) \right ) \\
    \verbose{
    =\ &\Pstar \left ( \max_{\Dcal(n-1) < k \leq \Dcal(n)} \left \lvert \sum_{i=\Dcal(n-1)+1}^k (\widetilde X_i - \widetilde Y_i)  \right \rvert > 2z \log \Dcal(n-1) \right ) \\
    }
    =\ &\Ptil_n \left ( \max_{\Dcal(n-1) < k \leq \Dcal(n)} \left \lvert \sum_{i=\Dcal(n-1)+1}^k \left (\widetilde X_{i}^\brackn - \widetilde Y_{i}^\brackn \right )  \right \rvert > 2z \log \Dcal(n-1) \right ) \\
    \leq\ &\Ptil_n \left ( \max_{\Dcal(n-1) < k \leq \Dcal(n)} \left \lvert \sum_{i=\Dcal(n-1)+1}^k \left (\widetilde X_{i}^\brackn - \widetilde Y_{i}^\brackn \right )  \right \rvert > z \log d(n) \right ), 
  \end{align}
  where the final inequality follows from the fact that $\Dcal(n-1) \geq \sqrt{d(n)}$ for any $n$ since $\Dcal(n-1) \geq d(n-1) = \sqrt{d(n)}$.
  Applying \eqref{eq:exponential-proof-sakhanenko} as in Step 1, we observe that
  \begin{equation}
    (\star) \leq \sum_{n=n_m}^\infty \left ( 1 + \lambda \sqrt{2^{2^n}} \sigma \right ) \left ( 2^{2^n} \right )^{-c\lambda z} \leq  \sum_{n=n_m}^\infty \left ( 1 + \sqrt{2^{2^n}} \right ) \left ( 2^{2^n} \right )^{-c\lambda z},
  \end{equation}
  where the second inequality follows from the fact that $\sigma \leq \lambda^{-1}$ (recall \eqref{eq:sakhanenko-parameter-variance-inequalities}).
  It remains to prove that the same upper bound holds for $(\dagger)$.

  \paragraph{Step 5: Obtaining the desired high-probability bound on $(\dagger)$}
  Using a telescoping sum, note that for any $n \geq n_m$, the sum $\Lambdatil_{\Dcal(n-1)}$ can be re-written as
  \begin{equation}
   \Lambdatil_{\Dcal(n-1)} = \sum_{j=1}^{n-1} \left ( \Lambdatil_{\Dcal(j)} - \Lambdatil_{\Dcal(j-1)} \right ).
  \end{equation}
  Therefore, for any $t > 0$ and any $0 < \lambda < \widebar \lambda$, it holds that
  \begin{align}
    \Pstar \left ( |\Lambdatil_{\Dcal(n-1)}| \geq t \right ) &\leq \Pstar \left ( \sum_{j=1}^{n-1} \max_{\Dcal(j-1) < k \leq \Dcal(j)} \left \lvert \Lambdatil_k - \Lambdatil_{\Dcal(j-1)} \right \rvert \geq t \right )\\
    &\leq \EE_\Pstar \left [ \exp \left \{ c \lambda \sum_{j=1}^{n-1} \max_{\Dcal(j-1) < k \leq \Dcal(j)} \left \lvert \Lambdatil_k - \Lambdatil_{\Dcal(j-1)} \right \rvert \right \}  \right ] \cdot \exp \left \{ -c\lambda t \right \}.
  \end{align}
  Notice that the maxima $\max_{\Dcal(j-1) < k \leq \Dcal(j)} \left \lvert \Lambdatil_k - \Lambdatil_{\Dcal(j-1)} \right \rvert$ are independent across $j \in \{1, \dots, n-1\}$, and hence the expectation in the above inequality can be upper-bounded as
  \begin{align}
    \EE_\Pstar \left [ \exp \left \{ c \lambda \sum_{j=1}^{n-1} \max_{\Dcal(j-1) < k \leq \Dcal(j)} \left \lvert \Lambdatil_k - \Lambdatil_{\Dcal(j-1)} \right \rvert \right \}  \right ]
    &= \prod_{j=1}^{n-1} \EE_{\Ptil_j} \left[ \exp \left \{ c \lambda \max_{\Dcal(j-1) < k \leq \Dcal(j)} \left \lvert \Lambdatil_k - \Lambdatil_{\Dcal(j-1)} \right \rvert \right \}  \right ]\\
    &\leq \prod_{j=1}^{n-1}\left ( 1 + \lambda\sigma \sqrt{d(j)}  \right ).
  \end{align}
  Therefore, for any $t > 0$,
  \begin{align}
    \Pstar \left ( |\Lambdatil_{\Dcal(n-1)} | \geq t \right ) &\leq \exp \left \{ -c\lambda t \right \}\prod_{j=1}^{n-1} \left ( 1 + \lambda \sigma  \sqrt{d(j)} \right ) \leq \exp \left \{ -c\lambda t \right \}\prod_{j=1}^{n-1} \left ( 1 + \sqrt{d(j)} \right ). \label{eq:final markov step 5}
  \end{align}
  Recalling that $\sqrt{d(j)} = 2^{2^{j-1}}$, it is straightforward to check via induction that
  \begin{equation}\label{eq:induction identity}
    \prod_{j=1}^{n-1} \left ( 1 + \sqrt{d(j)} \right ) = \sqrt{2^{2^{n}}} - 1.
  \end{equation}
  Using the identity \eqref{eq:induction identity}, setting $t = 2z \log \Dcal(n-1)$ in \eqref{eq:final markov step 5}, and using the inequality $\Dcal(n-1) \geq \sqrt{d(n)}$ as in Step 4, we finally obtain
  \begin{equation}
    (\dagger) \leq \sum_{n=n_m}^\infty \left ( \sqrt{2^{2^n}} - 1 \right ) \left (2^{  2^n } \right )^{-c\lambda z}.
  \end{equation}
  Putting steps 3--5 together, we have that
  \begin{equation}
    \Pstar \left ( \supkm \left \lvert \frac{\Lambdatil_k}{\log k} \right \rvert > 4z \right ) \leq 2\sum_{n=n_m}^\infty \left (1 + \sqrt{2^{2^{n}}} \right ) \left (2^{  2^n } \right )^{-c\lambda z}.
  \end{equation}
  Taking $\lambda \to \widebar\lambda$ completes the proof.
\end{proof}

\subsection{Proof of \cref*{theorem:uniform-kmt-exponential}}\label{proof:uniform-kmt-exponential}
\begin{proof}[Proof of \cref{theorem:uniform-kmt-exponential}]
  Let us begin with the proof of sufficiency, i.e.~that Sakhanenko regularity implies the distribution-uniform KMT approximation in \eqref{eq:kmt-exponential-convergence}.
  Letting $X$ be a ($\ubar \lambda$, $\Acal$)-Sakhanenko regular random variable on the probability spaces $\probspaces$ with uniform variance lower bound $\ubar \sigma^2$,
  fix an arbitrary positive constant $\delta > 0$ and put $C_\delta := (1/2 + \delta) / (c \ubar \lambda)$ where $c \equiv c_{\ref*{theorem:concentration-kmt}}$. By \cref{theorem:concentration-kmt}, there exists for each $\alphain$ a construction $\probspacealphatil$ so that for any $m \geq 5$,
  \begin{align}
    \Ptil_\alpha \left ( \supkm \left \lvert \frac{\Lambdatil_k}{\log k} \right \rvert > 4C_\delta  \right ) \leq 2 \sum_{n=n_m}^\infty \frac{ 1+\sqrt{2^{2^n}} }{\left ( 2^{2^n} \right )^{1/2 + \delta}},
  \end{align}
  observing that the right-hand side no longer depends on $\alpha$.
  Taking suprema over $\alphain$ and limits as $m \to \infty$, we obtain
  \begin{equation}
    \lim_\mto \supalpha \Ptil_\alpha \left ( \supkm \left \lvert \frac{\Lambdatil_k}{\log k} \right \rvert >  \frac{2 + 4\delta}{c \ubar \lambda}  \right ) = 0.
  \end{equation}
  Instantiating the above with $\delta = 1/4$ (for instance), we have the desired result with $c_{\ref*{theorem:uniform-kmt-exponential}} := 3 / c_{\ref*{theorem:concentration-kmt}}$. This completes the first half of the proof.

  Moving on to the proof of necessity, suppose that $X$ is not Sakhanenko regular meaning that either
  \begin{equation}
   \ubar \lambda \equiv \inf_\alphain \lambda(P_\alpha)= 0\quad\text{or}\quad \supalpha \Var_\Palpha(X) = \infty
  \end{equation}
  where $\ubar \lambda$ and $\lambda(\cdot)$ are both given in \cref{definition:sakhanenko-regularity}.
  Our aim is to show that for every collection of constructions $\probspacestilde$ and every constant $C > 0$,
  \begin{equation}
    \lim_\mto \supalpha \Ptil_\alpha \left ( \supkm \left \lvert \frac{\Lambdatil_k}{\log k} \right \rvert > C \right ) > 0.
  \end{equation}
  Indeed, let $C > 0$ be arbitrary and notice that for any $\alphain$,
  \begin{align}
    \underbrace{\Palpha \left ( \sup_{k \geq m}\frac{| X_k |}{\log k}>  4C \right )}_{(\star)} &= \Ptil_\alpha \left (\supkm  \frac{ | \widetilde X_k |}{\log k}>  4C \right ) \\
    &\leq \underbrace{\Palphatil \left ( \sup_{k \geq m} \frac{1}{\log k} \left \lvert \widetilde X_k - \widetilde Y_k \right \rvert >  2C \right )}_{(\dagger)} + \underbrace{\Palphatil \left ( \supkm \frac{1}{\log k} \left \lvert \widetilde Y_k \right \rvert >  2C \right )}_{(\dagger\dagger)},
  \end{align}
  where we have used the fact that $\widetilde X$ has the same distribution under $\Palphatil$ as $X$ does under $\Palpha$. Now, notice that $(\dagger)$ can be upper bounded as
  \begin{align}
    (\dagger) &= \Ptil_\alpha \left ( \sup_{k \geq m} \frac{1}{\log k} \left \lvert \widetilde X_k - \widetilde Y_k \right \rvert >  2C \right ) \\
    &\leq \Ptil_\alpha \left ( \sup_{k \geq m} \frac{1}{\log k} \left ( \left \lvert \sum_{i=1}^k (\widetilde X_i - \widetilde Y_i) \right \rvert + \left \lvert \sum_{i=1}^{k-1} (\widetilde X_i - \widetilde Y_i) \right \rvert \right )  > 2C \right ) \\
    \verbose{
    &\leq \Ptil_\alpha \left ( \sup_{k \geq m} \frac{1}{\log k} \left \lvert \sum_{i=1}^k (\widetilde X_i - \widetilde Y_i)  \right \rvert > C \right ) + \Ptil_\alpha \left ( \sup_{k \geq m-1} \frac{1}{\log k} \left \lvert \sum_{i=1}^k (\widetilde X_i - \widetilde Y_i)  \right \rvert > C \right )\\
    }
    &\leq 2\Ptil_\alpha \left ( \sup_{k \geq m-1} \left \lvert \frac{\Lambdatil_k}{\log k} \right \rvert > C \right ).
  \end{align}
  Putting the above together with the former upper bound on $(\star)$ and taking suprema over $\alphain$, we notice that
  \begin{equation}
    \underbrace{\supalpha \Palpha \left ( \sup_{k \geq m}\frac{| X_k |}{\log k}>  4C \right )}_{\gamma_X(m)}  \leq \underbrace{2\supalpha \Palphatil \left ( \sup_{k \geq m-1} \left \lvert \frac{\Lambdatil_k}{\log k} \right \rvert > C \right )}_{\gamma_{\Lambda}(m)} + \underbrace{\supalpha \Palphatil \left ( \supkm \frac{1}{\log k} \left \lvert \widetilde Y_k \right \rvert >  2C \right )}_{\gamma_Y(m)}
  \end{equation}
  In what follows, we will show that the left-hand side $\gamma_X(m)$ does not vanish as $m \to \infty$ and that the second term in the right-hand side $\gamma_Y(m)$ does vanish as $m \to \infty$, and thus $\lim_\mto \gamma_\Lambda(m) > 0$, which will complete the proof.
\paragraph{Showing that $\displaystyle \lim_\mto \gamma_X(m) > 0$}
By the uniform second Borel-Cantelli lemma \citep[Lemma 2]{waudby2024distribution}, it suffices to show that
  \begin{equation}\label{eq:exponential-iff-desideratum}
    \lim_\mto \sup_\alphain \sum_{k=m}^\infty \Palpha \left ( (4C)^{-1} | X| > \log k \right ) > 0.
  \end{equation}
  Indeed, observe that for any $\alphain$ and any $m \geq 1$,
  \begin{align}
    \sum_{k=m}^\infty \Palpha \left ( (4C)^{-1} |X| > \log k \right ) &= \sum_{k=m}^\infty \Palpha \left ( \exp \left \{ (4C)^{-1} | X| \right \} > k \right ),
  \end{align}
  and thus by \citep[Lemma 10]{waudby2024distribution} (see also \citep{hu2017note}), \eqref{eq:exponential-iff-desideratum} holds if and only if $\exp \left \{ (4C)^{-1} |X| \right \}$ is not $\Acal$-uniformly integrable, that is,
  \begin{equation}\label{eq:exponential-moment-not-uniformly-integrable}
    \lim_\Kto \sup_\alphain \EE_\Palpha \left ( \exp \left \{ (4C)^{-1} |X| \right \} \1 \{ \exp \left \{ (4C)^{-1} |X| \right \} \geq K \} \right ) > 0.
  \end{equation}
  Indeed, using the assumption that $X$ is not Sakhanenko regular and invoking \cref{proposition:sakhanenko-param-bernstein-uniform-integrability}, we have that $\exp \left \{ \lambda |X| \right \}$ is not uniformly integrable for any $\lambda > 0$, implying that the inequality in \eqref{eq:exponential-moment-not-uniformly-integrable} holds. This completes the argument that $\lim_\mto \gamma_X(m) > 0$.

\paragraph{Showing that $\displaystyle \lim_\mto \gamma_Y(m) = 0$}
Writing out $\gamma_Y(m)$ for any $m \geq 2$ and union bounding, we have that
  \begin{align}
    \supalpha \Palphatil \left ( \supkm \frac{1}{\log k} \left \lvert \widetilde Y_k \right \rvert >  2C \right ) &\leq \supalpha \sum_{k = m}^\infty \Palphatil \left ( \lvert \widetilde Y_k \rvert >  2C \log k \right ).
  \end{align}
  Using a Chernoff bound, we note that $\Ptil_\alpha (|\Ytil_k| \geq y) \leq 2 \exp \left \{ -y^2/(2\sigma_\Palpha^2) \right \}$ for any $y > 0$ where $\sigma_\Palpha^2 = \Var_\Palpha(X)$  and hence
\begin{equation}
    \supalpha \Palphatil \left ( \supkm \frac{1}{\log k} \left \lvert \widetilde Y_k \right \rvert >  2C \right ) \leq \supalpha \sum_{k=m}^\infty 2\exp \left \{ - 2 C^2 \log^2 k / \widebar \sigma^2 \right \}
\end{equation}  
where $\widebar \sigma^2 < \infty$ is an upper bound on $\sup_\alphain \Var_\Palpha(X)$.
  Noting that the sum in the right-hand side no longer depends on $\alphain$ and is finite for $m = 1$, we have that the right-hand side vanishes as $m \to \infty$, which completes the argument that $\lim_\mto \gamma_Y(m) = 0$.
\paragraph{Concluding that $\displaystyle \lim_\mto \gamma_\Lambda(m) > 0$}
Putting the previous two paragraphs together combined with the fact that $\gamma_X(m) \leq \gamma_\Lambda(m) + \gamma_Y(m)$, we have that
  \begin{equation}
    \lim_\mto \supalpha \Palphatil \left ( \supkm \left \lvert \frac{\Lambdatil_k}{\log k} \right \rvert > C \right ) > 0,
  \end{equation}
  which completes the proof of \cref{theorem:uniform-kmt-exponential}.
\end{proof}

\section{Proofs from \cref*{section:power-moments}}\label{section:proofs-power-moments}

\subsection{Proof of \cref*{theorem:concentration-power-moments}}\label{proof:concentration-power-moments}

The proofs that follow rely on an inequality due to \citet{sakhanenko1985estimates,sakhanenko2006estimates} for finite collections of random variables. 
  \begin{lemma}[Sakhanenko's polynomial moment inequality \citep{sakhanenko1985estimates,sakhanenko2006estimates}]\label{lemma:sakhanenko-1d}
    Let $(X_1, \dots, X_n)$ be independent mean-zero random variables on a probability space $(\Omega, \Fcal, P)$ and let $q > 2$. One can construct a new probability space $(\widetilde \Omega, \widetilde \Fcal, \widetilde P)$ rich enough to contain the tuples $(\Xtil_i, \Ytil_i)_{i=1}^n$ so that $(X_1, \dots, X_n)$ and $(\Xtil_1,\dots, \Xtil_n)$ have the same law and $(\Ytil_1,\dots, \Ytil_n)$ are mean-zero independent Gaussian random variables with $\Var_{\widetilde P}(\widetilde Y_k) = \Var_P(X_k)$ for each $k \in \{1,\dots,n\}$ so that
    \begin{equation}
      \EE_{\widetilde P} \left ( \max_{1 \leq k \leq n} \left \lvert \sum_{i=1}^k \widetilde X_i - \sum_{i=1}^k \widetilde Y_i \right \rvert  \right )^q \leq C_S(q) \sum_{i=1}^n \EE_P |X_i|^q , 
    \end{equation}
    where $C_S(q) > 0$ is a constant depending only on $q$.
  \end{lemma}
  Note that Sakhanenko's inequality (\cref{lemma:sakhanenko-1d}) applies to finite collections $(X_1, \dots, X_n)$ of random variables and in particular, the construction $\probspacetilde$ may itself depend on $n$. However, the statement of \cref{theorem:concentration-power-moments} involves a single construction containing the infinite sequences $\infseqn{\Xtil_n}$ and $\infseqn{\Ytil_n}$. As such, the following proof will partition $\NN$ into finite collections and apply \cref{lemma:sakhanenko-1d} on each of them, ultimately combining these constructions into one in a manner similar to that found in the proof of \cref{theorem:concentration-kmt}. The ideas behind the aforementioned partitioning argument are inspired in part by a proof found in the lecture notes of \citet[Theorem 3.3]{lifshits2000lecture} which lifts Sakhanenko's inequality (\cref{lemma:sakhanenko-1d}) to a common probability space for all $n \in \NN$. However, we will not make use of Lifshits' result directly as it is still an asymptotic and distribution-pointwise statement that is insufficient for our purposes.

  The proof of \cref{theorem:concentration-power-moments} additionally relies on a maximal weighted sum inequality which we state and prove as follows.

  \begin{lemma}[A maximal weighted sum inequality]\label{lemma:summation-by-parts}
    Let $\infseqn{a_n}$ be a monotonically nondecreasing and positive sequence and let $\infseqn{b_n}$ be any real sequence. Then for any integers $m$ and $K$ such that $m < K$, we have
  \begin{equation}
    \max_{m < k \leq K} \frac{|\sum_{i=m+1}^k b_i|}{a_k} \leq 2 \max_{m < k \leq K} \left \lvert \sum_{i=m+1}^k\frac{b_i}{a_i} \right \rvert.
  \end{equation}
  \end{lemma}
  \begin{proof}
  For each $n \in \NN$, define $b_n'$ as $b_n' = b_n / a_n$
  and for every pair of positive integers $m, k$ for which $m < k$, define $S_\mk'$ and $S_\mk$ as
  \begin{equation}
    S_\mk' := \sum_{i=m+1}^k b_i'\quad\text{and}\quad S_\mk := \sum_{i=m+1}^k b_i,
  \end{equation}
  and define both as 0 whenever $m = k$.
  Note that by construction,
  \begin{equation}
    S_\mk = \sum_{i=m+1}^k b_i = \sum_{i=m+1}^k a_i b_i'.
  \end{equation}
  Using summation by parts, $S_\mk$ can be written as
  \begin{align}
    S_\mk &= \sum_{i=m+1}^k a_i b_i' = \sum_{i=m+1}^k a_i (S_{m, i}' - S_{m,i-1}') = a_k S_{m,k}' - \sum_{i=m+1}^k (a_i-a_{i-1}) S_{m,i-1}'.
  \end{align}
  Therefore, we have that
  \begin{align}
    |S_\mk| \leq a_k |S_\mk'| + \max_{m < j \leq k}|S_{m, j}'|\sum_{i=m+1}^k (a_i-a_{i-1}) \leq 2a_k \max_{m < j \leq k} |S_\mj'|,
  \end{align}
  and hence
  \begin{equation}
    \frac{|S_\mk|}{a_k} \leq 2 \max_{m < j \leq k } |S_\mj'|.
  \end{equation}
  In particular, for any $K$, we have
  \begin{equation}
    \max_{m < k \leq K} \frac{|S_\mk|}{a_k} \leq 2 \max_{m < k \leq K} \max_{m < j \leq k} |S_\mj'| = 2 \max_{m < k \leq K} |S_\mk'|.
  \end{equation}
  This completes the proof of \cref{lemma:summation-by-parts}.
  \end{proof}

  With \cref{lemma:sakhanenko-1d,lemma:summation-by-parts} in hand, we now prove \cref{theorem:concentration-power-moments}.

\begin{proof}[Proof of \cref{theorem:concentration-power-moments}]
  Let $\infseqn{\ubar a_n}$ and $\infseqn{a_n}$ be the sequences described in the statement of \cref{theorem:concentration-power-moments} and define
  \begin{equation}
    U := \sum_{k=1}^\infty \frac{\EE_P |X_k|^q}{\ubar a_k^q} < \infty.
  \end{equation}
  Using $U$, partition $\NN$ into blocks $\{ \Ncal_b \}_{b \in \NN}$ given by
  \begin{equation}
    \Ncal_b := \left \{ n \in \NN : 2^{-b} U < \sum_{k=n}^\infty \frac{\EE_P|X_k|^q}{\ubar a_k^q} \leq 2^{-b + 1} U \right \},
  \end{equation}
  noticing that $\bigcup_{b \in \NN} \Ncal_b = \NN$ and $|\Ncal_b| < \infty$ since $\liminf_\nto \Var_P(X_n) > 0$.  
  For each $b \in \NN$, we invoke \cref{lemma:sakhanenko-1d} as it applies to $(X_n / \ubar a_n)_{n \in \Ncal_b}$ to obtain $(\Xtil_n^\brackb / \ubar a_n)_{n \in \Ncal_b}$ and $(\Ytil_n^\brackb/ \ubar a_n)_{n \in \Ncal_b}$ on the space $(\widetilde \Omega^\brackb, \widetilde \Fcal^\brackb, \widetilde P^\brackb)$ with the property that
  \begin{equation}
    \EE_{\widetilde P^\brackb} \left ( \max_{n \in \Ncal_b} \left \lvert \sum_{k \in \Ncal_b : k \leq n} \frac{\Xtil_k^\brackb - \Ytil_k^\brackb}{\ubar a_k} \right \rvert^q \right ) \leq C_S(q) U_b,\quad\text{where}\quad U_b := \sum_{k \in \Ncal_b} \frac{\EE_P|X_k|^q}{\ubar a_k^q},
  \end{equation}
  and where $C_S(q) > 0$ is the same constant as in \cref{lemma:sakhanenko-1d} that depends on $q$. We take the maximum to be 0 if the set $\Ncal_b$ is empty.
  Define the sequences $\infseqn{\Xtil_n}$ and $\infseqn{\Ytil_n}$ on a common space $(\widetilde \Omega, \widetilde \Fcal, \Ptil)$ so that they agree with $(\Xtil_n^\brackb)_{n \in \Ncal_b}$ and $(\Ytil_n^\brackb)_{n \in \Ncal_b}$, respectively for each $b \in \NN$. Concretely, define $\widetilde \Omega := \prod_{b=1}^\infty \widetilde \Omega^\brackb$ and obtain the associated sigma-algebra $\widetilde \Fcal$ and probability measure $\Ptil$ by Kolmogorov's extension theorem and define the random variables $\Xtil_i$ and $\Ytil_i$ for each $\omega \in \widetilde \Omega$ as
  \begin{align}
    \Xtil_k(\omega) &:= \Xtil_{k}^\brackb \equiv \Xtil_{k}^\brackb(\omega_{\ubar N_b}, \dots, \omega_{\widebar N_b})\quad\text{and}\\
    \Ytil_k(\omega) &:= \Ytil_{k}^\brackb \equiv \Ytil_{k}^\brackb(\omega_{\ubar N_b}, \dots, \omega_{\widebar N_b}) \quad\text{whenever } k \in \Ncal_b,
  \end{align}
  where $\ubar N_b := \min \Ncal_b$ and $\widebar N_b := \max \Ncal_b$ are the smallest and largest integers in $\Ncal_b$, respectively. If the sets are empty, the quantities $\ubar N_b$ and $\widebar N_b$ are of no consequence to the arguments that follow so they can be trivially defined as $\ubar N_b = \widebar N_b = 0$ (along with $\Xtil_0 = \Ytil_0 = 0$ and $a_0 = \ubar a_0 = 1$ to avoid undefinedness of terms).
  For each $k \in \NN$, define
  \begin{equation}
    \Lambdatil_k := \sum_{i=1}^k (\Xtil_i - \Ytil_i)
  \end{equation}
  and for each $m \in \NN$, let $b(m)$ be the index for which $m \in \Ncal_{b(m)}$. We will write $n_m := \ubar N_{b(m)}$ to ease notation for the time being. Since $m \geq n_m$ for every $m \in \NN$ by construction, we have that
  \begin{align}
    \Ptil \left ( \supkm \frac{|\Lambdatil_k|}{a_k} \geq \eps \right ) &\leq \Ptil \left ( \sup_{k \geq n_m} \left \lvert \frac{\Lambdatil_k}{a_k} - \frac{\Lambdatil_{n_m}}{a_k} + \frac{\Lambdatil_{n_m}}{a_k} \right  \rvert \geq \eps \right ) \\
                                                                       &\leq \Ptil \left ( \sup_{k \geq n_m} \frac{|\Lambdatil_k - \Lambdatil_{n_m}|}{a_k} + \sup_{k\geq n_m} \frac{|\Lambdatil_{n_m}|}{a_k} \geq \eps \right ) \\
    &\leq \underbrace{\Ptil \left ( \sup_{k \geq n_m + 1} \frac{|\Lambdatil_k - \Lambdatil_{n_m}|}{\ubar a_k} \geq \eps/2 \right )}_{(\star)} + \underbrace{\Ptil \left ( \frac{|\Lambdatil_{n_m}|}{a_{n_m}} \geq \eps/2 \right )}_{(\dagger)},
  \end{align}
  where the last inequality used the inequality $\ubar a_n \leq a_n$ for each $n$ as well as monotonicity of $\infseqn{a_n}$ in $(\star)$ and $(\dagger)$, respectively.
  Let us now provide an upper bound on $(\star)$.
  \paragraph{Deriving an upper bound on $(\star)$}
  Writing out $(\star)$, applying Markov's inequality, and appealing to monotone convergence, we have that
  \begin{align}
    (\star) &= \Ptil \left ( \sup_{k \geq n_m+1} \frac{|\Lambdatil_k-\Lambdatil_{n_m}|}{\ubar a_k} \geq \eps/2 \right ) \\
    \verbose{
    &= \Ptil \left ( \sup_{k \geq n_m +1 } \frac{| \sum_{i=n_m + 1}^k (\Xtil_i - \Ytil_i) | }{\ubar a_k} \geq \eps/2 \right )\\
    }
            &\leq \frac{2^q}{\eps^q}\lim_{M\to \infty}  \EE_\Ptil \left ( \max_{n_m < k \leq M}\frac{| \sum_{i=n_m + 1}^k (\Xtil_i - \Ytil_i) |^q }{\ubar a_k^q}   \right ).
  \end{align}
  Notice that since $\liminf_{\nto} \Var_P(X_n) > 0$, we have that $\liminf_{\nto} \EE_P|X_n|^q > 0$. Combined with the fact that $U < \infty$, we have that $\widebar N_{b(M)} \to \infty$ as $M\to\infty$. Therefore, we can re-write the above with $M$ replaced by $\widebar N_{b(M)}$:
  \begin{equation}
    (\star) \leq \frac{2^q}{\eps^q}\lim_{M\to \infty}  \EE_\Ptil \left ( \max_{n_m < k \leq \widebar N_{b(M)}}\frac{| \sum_{i=n_m + 1}^k (\Xtil_i - \Ytil_i) |^q }{\ubar a_k^q}   \right ) 
  \end{equation}
  Moving forward, we will focus on bounding the expectation inside the above limit for an arbitrary $M \geq m$. 
  Indeed, applying \cref{lemma:summation-by-parts}, we have that the aforementioned expectation can be bounded as
  \begin{align}
    \EE_\Ptil \left ( \max_{n_m < k \leq \widebar N_{b(M)}}\frac{| \sum_{i=n_m + 1}^k (\Xtil_i - \Ytil_i) |^q }{\ubar a_k^q}   \right ) &\leq  2^q \EE_\Ptil \left ( \max_{\ubar N_{b(m)} \leq k \leq \widebar N_{b(M)}} \left \lvert \sum_{i=\ubar N_{b(m)}}^k \frac{\Xtil_i - \Ytil_i}{\ubar a_i} \right \rvert^q \right ).\label{eq:polynomial-proof-expectation-to-bound}
  \end{align}
where we have additionally re-written $n_m$ as it was defined previously by $\ubar N_{b(m)}$.
  Define the sum $\Lambdatil'_{\ubar N_b, k}$ inside the absolute value of the right-hand side and its maximum absolute value $\Lambdatil'_{\bmax}$ as
  \begin{equation}
    \Lambdatil'_{\ubar N_b, k} := \sum_{i=\ubar N_b}^k \frac{\Xtil_i - \Ytil_i}{\ubar a_i}\quad\text{and}\quad\Lambdatil'_\bmax := \max_{\ubar N_b \leq k \leq \widebar N_{b}} |\Lambdatil_{\ubar N_b, k}'|
  \end{equation}
  for any $k$ and $b$ so that $\ubar N_b \leq k$.
  By the triangle inequality
and appealing to the notation introduced above, we have that the expectation in \eqref{eq:polynomial-proof-expectation-to-bound} can be re-written and upper bounded as
  \begin{align}
    \EE_\Ptil \left ( \max_{\ubar N_{b(m)} \leq k \leq \widebar N_{b(M)}} \left \lvert \Lambdatil'_\Nbmk  \right \rvert^q \right ) &\leq \EE_\Ptil \left [ \left ( \max_{\ubar N_{b(m)} \leq k \leq \widebar N_{b(M)}}  \left \lvert \sum_{b=b(m)}^{b(k) - 1} \Lambdatil'_{\ubar N_b, \widebar N_b}  + \Lambdatil'_{\ubar N_{b(k)}, k} \right \rvert  \right )^q \right ]\\
    \verbose{
                                              &\leq \EE_\Ptil \left [ \left ( \max_{\ubar N_{b(m)} \leq k \leq \widebar N_{b(M)}} \left ( \sum_{b=b(m)}^{b(k) - 1}  \left \lvert  \Lambdatil'_{\ubar N_b, \widebar N_b}  \right \rvert + \left \lvert \Lambdatil'_{\ubar N_{b(k)}, k} \right \rvert \right )  \right )^q \right ]\\
    }
                                              &\leq \EE_\Ptil \left [ \left ( \sum_{b=b(m)}^{b(M)} \max_{\ubar N_b \leq k \leq \widebar N_b}  \left \lvert  \Lambdatil'_{\ubar N_b, k}  \right \rvert  \right )^q \right ] \equiv \EE_\Ptil  \left \lvert \sum_{b=b(m)}^{b(M)} \Lambdatil'_\bmax \right \rvert^q.
  \end{align}
   Using the fact that $\EE_P |W + Z|^q \leq 2^{q-1} \left ( \EE_P |W|^q + \EE_P |Z|^q \right )$ for any two random variables $W,Z$ with finite $q^\tth$ moments, we have that
  \begin{equation}
    \EE_\Ptil \left \lvert  \sum_{b=b(m)}^{b(M)} \Lambdatil'_{b,\Max} \right \rvert^q \leq 2^{q-1} \underbrace{\EE_\Ptil  \left \lvert  \sum_{b=b(m)}^{b(M)} \left ( \Lambdatil'_{b,\Max} -  \EE_\Ptil \Lambdatil'_{b,\Max}\right ) \right \rvert^q}_{(\star i)} + 2^{q-1} \underbrace{\left \lvert \sum_{b=b(m)}^{b(M)}\EE_\Ptil \Lambdatil'_{b, \Max} \right \rvert^q}_{(\star ii)},
  \end{equation}
  and we will focus on bounding $(\star i)$ and $(\star ii)$ separately. Starting with the former, by Rosenthal's inequality \citep{rosenthal1970subspaces,rosenthal1972span}, there exists a constant $C_R(q) > 0$ only depending on $q$ so that $(\star i)$ is upper bounded as follows:
  \begin{align}
    (\star i)&\leq C_R(q) \sum_{b=b(m)}^{b(M)} \EE_\Ptil \left \lvert \Lambdatil'_{b,\Max} - \EE_\Ptil \Lambdatil'_{b,\Max} \right \rvert^q + C_R(q) \left (\sum_{b=b(m)}^{b(M)} \EE_\Ptil   \left (\Lambdatil'_{b, \Max} - \EE_\Ptil \Lambdatil'_{b, \Max} \right )^2 \right )^{q/2}\\
    &\leq 2^q C_R(q) \sum_{b=b(m)}^{b(M)} \EE_\Ptil \left \lvert \Lambdatil'_{b,\Max} \right \rvert^q + 2^q C_R(q) \left (\sum_{b=b(m)}^{b(M)} \EE_\Ptil   \left (\Lambdatil'_{b, \Max}  \right )^2 \right )^{q/2}.
  \end{align}
  Notice that by our initial constructions of $\Ncal_b$ and applications of \cref{lemma:sakhanenko-1d} performed at the outset of the proof, the sum in the first term in the right-hand side of the above inequality is upper-bounded as
  \begin{align}
    \sum_{b=b(m)}^{b(M)} \EE_\Ptil \left \lvert \Lambdatil'_{b, \Max} \right \rvert^q &\leq C_S(q)\sum_{b=b(m)}^{b(M)} U_b \\
                                                                                      &= C_S(q)\sum_{b=b(m)}^{b(M)} \sum_{k \in \Ncal_b} \frac{\EE_P |X_k|^q}{\ubar a_k^q} \\
    &= C_S(q)\sum_{k=\ubar N_{b(m)}}^{\widebar N_{b(M)}} \frac{\EE_P |X_k|^q}{\ubar a_k^q} \leq C_S(q)U 2^{1-b(m)}.
  \end{align}
  Turning now to the second term in the right-hand side of the upper bound on $(\star i)$---ignoring the constant $2^q C_R(q)$ out front for now---we have by Jensen's inequality,
  \begin{align}
    \left (\sum_{b=b(m)}^{b(M)} \EE_\Ptil  \left ( \Lambdatil'_\bmax \right )^2 \right )^{q/2} \leq \left (\sum_{b=b(m)}^{b(M)} \left (\EE_\Ptil  \left \lvert \Lambdatil'_\bmax \right \rvert^q \right )^{2/q} \right )^{q/2}
                                                     \leq C_S(q)\left ( \sum_{b=b(m)}^{b(M)} \left ( U_b \right )^{2/q} \right )^{q/2}.
  \end{align}
  By construction, note that $U_b \leq U 2^{-b+1}$ and hence we can further upper bound the above quantity as
  \begin{align}
                                                     \left ( \sum_{b=b(m)}^{b(M)} \left ( U_b \right )^{2/q} \right )^{q/2}                                                                             &\leq 2U \left ( \sum_{b=b(m)}^{b(M)} 2^{-2b/q} \right )^{q/2} \\
                                                     \verbose{                                                                                 &\leq  2U \left ( \sum_{b=b(m)}^\infty 2^{-(2/q)b} \right )^{q/2} \\}
                                                                                                                                      &\leq 2U \left ( \int_{b(m)-1}^\infty 2^{-(2/q) x} dx \right )^{q/2} \\
                                                                                                                                      &= 2U \left ( \frac{2^{-(2/q)(b(m) - 1)}}{2\log 2 / q} \right )^{q/2} \\
    &= \left ( \frac{q}{2\log 2} \right )^{q/2} U 2^{-b(m) + 2}.
  \end{align}
  Putting the previous two upper bounds together, we have the following upper bound on $(\star i)$:
  \begin{equation}
    (\star i) \leq 2^{q+1} C_R(q) C_S(q) U 2^{-b(m)} + 2^{q+2}C_R(q) C_S(q) \left ( \frac{q}{2\log 2} \right )^{q/2} U 2^{-b(m)}  .
  \end{equation}
  Focusing now on $(\star ii)$, we have through a similar argument that by Jensen's inequality,
  \begin{align}
    (\star ii) &= \left \lvert \sum_{b=b(m)}^{b(M)} \EE_\Ptil \Lambdatil'_\bmax \right \rvert^q \\
    \verbose{
               &= \left \lvert \sum_{b=b(m)}^{b(M)} \EE_\Ptil \left ( |\Lambdatil'_\bmax|^q \right )^{1/q} \right \rvert^q \\
               &\leq \left \lvert \sum_{b=b(m)}^{b(M)} \left (\EE_\Ptil  |\Lambdatil'_\bmax|^q \right )^{1/q} \right \rvert^q \\
    }
               &\leq C_S(q)\left \lvert \sum_{b=b(m)}^{b(M)} \left ( U_b \right )^{1/q} \right \rvert^q.
  \end{align}
  Once again using the fact that $U_b \leq 2^{-b+1} U$, we have that
  \begin{align}
    (\star ii) \leq 2C_S(q)U \left \lvert \int_{b(m)-1}^\infty 2^{-x/q}dx \right \rvert^q = 4 C_S(q) U \left ( \frac{q}{\log 2} \right )^q 2^{-b(m)}.
  \end{align}
  Putting the bounds on $(\star i)$ and $(\star ii)$ together and consolidating constants only depending on $q$ into $C_{\ref*{theorem:concentration-power-moments}}$ (and allowing this constant to vary from one calculation to the next as long as those changes only depend on $q$), we have the following upper bound on $(\star)$:
  \begin{align}
    (\star) &\leq \frac{C_{\ref*{theorem:concentration-power-moments}}}{\eps^q} \lim_\Mto \left ( \sum_{k=\ubar N_{b(m)}}^\infty \frac{\EE_P|X_k|^q}{\ubar a_k^q} +  U  2^{-b(m)} \right ).
  \end{align}
  Noting that the expression inside the limit no longer depends on $M$, we have that
  \begin{equation}
    (\star) \leq \frac{C_{\ref*{theorem:concentration-power-moments}} 2^{-b(m)}}{\eps^q} \sum_{k=1}^\infty \frac{\EE_P |X_k|^q}{\ubar a_k^q} < \frac{C_{\ref*{theorem:concentration-power-moments}}}{\eps^q}\sum_{k=m}^\infty \frac{\EE_P|X_k|^q}{\ubar a_k^q},
  \end{equation}
  where the second inequality follows from the definition of $b(m)$ since $m \in \Ncal_{b(m)}$ precisely when $2^{-b(m)} U < \sum_{k=m}^\infty \EE_P|X_k|^q / \ubar a_k^q$.
  This completes the desired upper bound on $(\star)$. Let us now provide the required upper bound on $(\dagger)$.

  \paragraph{Deriving an upper bound on $(\dagger)$}
  By Markov's inequality,
  \begin{align}
    (\dagger) &= \Ptil \left ( \frac{|\Lambdatil_{n_m}|}{a_{n_m}} \geq \eps/2 \right ) \leq \frac{2^{q} }{\eps^q a_{n_m}^q}  \EE_\Ptil \left \lvert\Lambdatil_{n_m} \right \rvert^q
  \end{align}
  Focusing on the expectation, dividing by $\ubar a_{n_m}^q$, applying the triangle inequality, and then \cref{lemma:summation-by-parts}, we have
  \begin{align}
    \EE_\Ptil \left \lvert\frac{ \Lambdatil_{n_m}}{\ubar a_{n_m}} \right \rvert^q \leq 2^q \left ( \sum_{b=1}^{b(n_m)} \left ( \EE_\Ptil \left \lvert \Lambdatil'_{b,\Max} \right \rvert^q \right )^{1/q} \right )^q.
  \end{align}
  Applying \cref{lemma:sakhanenko-1d}, we have that the above can be re-written and upper bounded as
\begin{align}
  \EE_\Ptil \left \lvert\frac{ \Lambdatil_{n_m}}{\ubar a_{n_m}} \right \rvert^q &\leq 2^q \left ( \sum_{b=1}^{b(n_m)} \left ( \EE_{\Ptil^\brackb} \max_{k \in \Ncal_b}\left \lvert  \sum_{i \in \Ncal_b : i \leq k} \frac{\Xtil_i^\brackb - \Ytil_i^\brackb}{\ubar a_i} \right \rvert^q \right )^{1/q} \right )^q \\
                                                                                &\leq 2^{q} C_S(q)\left ( \sum_{b=1}^{b(n_m)} U_b^{1/q}  \right )^q\\
                                                                                &\leq 2^{q+1} U C_S(q) \left ( \sum_{b=1}^\infty 2^{-b/q}  \right )^q = C_{\ref*{theorem:concentration-power-moments}} U,
\end{align}
where $C_{\ref*{theorem:concentration-power-moments}}$ depends only on $q$.
  It follows that
  \begin{align}
    (\dagger) & \leq \frac{C_{\ref*{theorem:concentration-power-moments}}}{\eps^q a_{n_m}^q} \left ( \frac{\ubar a_{n_m}}{\ubar a_{n_m}}\right )^q  \EE_\Ptil \left \lvert\Lambdatil_{n_m} \right \rvert^q  \leq  \left ( \frac{\ubar a_{n_m}}{a_{n_m}} \right )^q\frac{C_{\ref*{theorem:concentration-power-moments}}}{\eps^q} U .
  \end{align}
  This completes the upper bound on $(\dagger)$.

  It remains to show that $n_m$ can be written as
  \begin{equation}
    n_m = \min \left \{ n \in \NN : \log_2 \left ( U / U_\geqn \right ) \geq \left \lfloor \log_2(U / U_\geqm)  \right \rfloor
    \right \}
  \end{equation}
  where for any $n \in \NN$, $U_\geqn := \sum_{k=n}^\infty \EE_P|X_k|^q / \ubar a_k^q$.
  Indeed, recall that $n_m$ is defined as $\min \Ncal_{b(m)}$ and by definition of $b(m)$, we have that
  \begin{equation}
    2^{-b(m)} U < U_\geqm \leq 2^{-b(m) + 1} U,
  \end{equation}
  and hence we have the inequality
  \begin{equation}
    \log_2(U / U_\geqm) < b(m) \leq \log_2 \left ( U / U_\geqm \right ) + 1,
  \end{equation}
  and since $b(m)$ is an integer, we have
  \begin{equation}
    b(m) = \left \lfloor \log_2 \left ( U / U_\geqm \right ) + 1 \right \rfloor.
  \end{equation}
  Putting the above together with the definition of $n_m = \min \Ncal_{b(m)}$, we have that
  \begin{align}
    n_m &= \min \Ncal_{b(m)}\\
    \verbose{
    &= \min \left \{ n \in \NN : 2^{-b(m)} U < U_\geqn \leq 2^{-b(m)+1} U
      \right \}\\
    }
    &= \min \left \{ n \in \NN : U_\geqn \leq 2^{-b(m)+1} U
      \right \}\\
    &= \min \left \{ n \in \NN : \log_2 (U_\geqn) \leq -b(m)+1 + \log_2 U
      \right \}\\
    \verbose{
    &= \min \left \{ n \in \NN : \log_2 (U_\geqn / U) \leq -\left \lfloor \log_2(U / U_\geqm) + 1 \right \rfloor
      +1 
      \right \}\\
    }
    &= \min \left \{ n \in \NN : \log_2 (U / U_\geqn) \geq \left \lfloor \log_2(U / U_\geqm) \right \rfloor 
      \right \},
  \end{align}
  which completes the proof of \cref{theorem:concentration-power-moments}.
\end{proof}

\subsection{Proof of \cref*{corollary:shao-strong-approx-uniform}}\label{proof:shao-strong-approx-uniform}

Before proving \cref{corollary:shao-strong-approx-uniform}, we introduce the following lemma which serves as a uniform generalization of \citet[Lemma 2.2]{shao1995strong}.

\begin{lemma}[An index-uniform generalization of Lemma 2.2 in \citet{shao1995strong}]\label{lemma:existence of slower sequence}
  Let $\Ical$ be an arbitrary index set.
  For every $i \in \Ical$, let $\infseqn{b_n^\bracki}$ be a nonnegative sequence of real numbers and let $\infseqn{a_n}$ be one that is nondecreasing and diverging. Suppose that
  \begin{equation}
    \lim_\mto \sup_{i \in \Ical} \sum_{k=m}^\infty \frac{b_k^\bracki}{a_k} = 0.
  \end{equation}
  Then there exists a nondecreasing and diverging sequence $\infseqn{\ubar a_n}$ such that $\ubar a_n \leq a_n$ for every $n \in \NN$ and $\ubar a_n / a_n \to 0$ monotonically so that
  \begin{equation}
    \lim_\mto \sup_{i \in \Ical} \sum_{k=m}^\infty \frac{b_k^\bracki}{\ubar a_k} = 0.
  \end{equation}
\end{lemma}

\begin{proof}
  Notice that by virtue of the fact that $\infseqn{a_n}$ is diverging, there is a subsequence $(n(k))_{k=1}^\infty$ of positive integers for which
  \begin{equation}\label{eq:subsequence doubling}
    a_{n(k+1)} \geq 2 a_{n(k)}
  \end{equation}
  and
  \begin{equation}\label{eq:existence-of-sequence-lemma-1/k3}
    \sup_{i \in \Ical}\sum_{j=n(k)}^\infty \frac{b_j^\bracki}{a_j} \leq \frac{1}{(k+1)^3}.
  \end{equation}
  We will define the sequence $\infseqm{v_m}$ in terms of the subsequence $(n(k))_{k=1}^\infty$ as follows.
  \begin{alignat}{3}
    &v(m) = 1  &&\quad\quad\text{for } 1 \leq m \leq n(1)  \\
    &v(m) = v(n(k)) + \frac{a_m - a_{n(k)}}{a_{n(k+1)}} &&\quad\quad \text{for } n(k) < m \leq n(k+1).
  \end{alignat}
  By definition of $v$, we notice that
  \begin{equation}
    v(n(k)) = 1 + \sum_{i=1}^{k-1} \left ( 1-\frac{a_{n(i)}}{a_{n(i+1)}} \right ),
  \end{equation}
  and since $a_n$ is nondecreasing, we have that $v(n(k)) \leq k$. Now, consider the following upper bound:
  \begin{align}
    \sup_{i \in \Ical} \sum_{k=n(j)}^\infty \frac{b_k^\bracki}{a_k / v(k)} &= \sup_{i \in \Ical} \sum_{\ell=j}^\infty \sum_{k=n(\ell)}^{n(\ell+1)-1} \frac{b_k^\bracki}{a_k / v(k)} \leq \sup_{i \in \Ical} \sum_{\ell=j}^\infty \sum_{k=n(\ell)}^{n(\ell+1)-1} \frac{b_k^\bracki}{a_k / v(n(\ell+1))}.
  \end{align}
  Using the fact that $v(n(k)) \leq k$ for each $k$ combined with the upper bound in \eqref{eq:existence-of-sequence-lemma-1/k3}, we have that
  \begin{align}
    \sup_{i \in \Ical} \sum_{k=n(j)}^\infty \frac{b_k^\bracki}{a_k / v(k)} &\leq \sum_{\ell=j}^\infty \frac{(\ell+1)}{(\ell+1)^3} \leq \int_{j-1}^\infty \frac{1}{(y+1)^2}\dd y = \frac{1}{j},
  \end{align}
  and hence we have that
  \begin{equation}
    \lim_\mto \sup_{i \in \Ical} \sum_{k=m}^\infty \frac{b_k^\bracki}{a_k/v(k)} = 0.
  \end{equation}
  By construction in \eqref{eq:subsequence doubling} combined with the definition of $v$, we have that $v(n(k)) \geq 1 + (k-1)/2 > k/2$ so $v(m) \to \infty$ as $m \to \infty$.
  Since $v(m) \geq 1$ for every $m$, we have the desired result for the sequence $\ubar a_m := a_m / v(m)$, completing the proof of \cref{lemma:existence of slower sequence}.
\end{proof}

\begin{proof}[Proof of \cref{corollary:shao-strong-approx-uniform}]
  By \cref{lemma:existence of slower sequence}, we have that under the assumptions of \cref{corollary:shao-strong-approx-uniform}, there exists a positive sequence $\infseqn{\ubar a_n}$ with the properties that $\ubar a_n \leq a_n$ for all $n$ and $\ubar a_n/a_n \to 0$ monotonically and so that
  \begin{equation}\label{eq:uniform boundedness and integrability polynomial moment corollary}
    \sup_\alphain \sum_{k=1}^\infty \frac{\EE_\Palpha|X_k|^q}{\ubar a_k^q} < \infty\quad\text{and}\quad \lim_\mto \sup_\alphain \sum_{k=m}^\infty \frac{\EE_\Palpha|X_k|^q}{\ubar a_k^q} = 0.
  \end{equation}
  Applying \cref{theorem:concentration-power-moments}, we have that there exists a collection of constructions so that for every $\alphain$ and $\eps > 0$,
  \begin{equation}
    \Palphatil \left ( \supkm \frac{|\sum_{i=1}^k (\Xtil_i - \Ytil_i)|}{a_k} \geq \eps \right ) \leq  \frac{C_{\ref*{theorem:concentration-power-moments}}(q)}{\eps^q} \left ( \sum_{k=m}^\infty \frac{\EE_\Palpha |X_k|^q}{\ubar a_k^q} + \frac{\ubar a_{n_m^\brackalpha}^q}{a_{n_m^\brackalpha}^q}\sum_{k=1}^\infty \frac{\EE_\Palpha|X_k|^q}{\ubar a_k^q} \right ),
  \end{equation}
  where $n_m^\brackalpha$ is defined as in \cref{theorem:concentration-power-moments} but now explicitly depending on $\alpha$:
 \begin{equation}
  n_m^\brackalpha = \min \left \{ n \in \NN : \log_2 \left ( \frac{\sum_{k=1}^\infty \EE_\Palpha|X_k|^q / \ubar a_k^q}{ \sum_{k=n}^\infty \EE_\Palpha|X_k|^q / \ubar a_k^q } \right ) \geq \left \lfloor \log_2 \left ( \frac{\sum_{k=1}^\infty \EE_\Palpha|X_k|^q / \ubar a_k^q}{ \sum_{k=m}^\infty \EE_\Palpha|X_k|^q / \ubar a_k^q } \right ) \right \rfloor
    \right \}. 
 \end{equation} 
 It suffices to show that $\supalpha (\ubar a_{n_m^\brackalpha} / a_{n_m^\brackalpha}) \to 0$ as $\mto$. Indeed, notice that by taking infima and suprema over $\alphain$ inside the definition of $n_m^\brackalpha$, we have that
 \begin{equation}
   n_m^\brackalpha \geq \ubar n_m := \min \left \{ n : \log_2 \left ( \frac{\sup_\alphain\sum_{k=1}^\infty \EE_\Palpha|X_k|^q / \ubar a_k^q}{ \inf_\alphain \sum_{k=n}^\infty \EE_\Palpha|X_k|^q / \ubar a_k^q } \right ) \geq \left \lfloor \log_2 \left ( \frac{\inf_\alphain\sum_{k=1}^\infty \EE_\Palpha|X_k|^q / \ubar a_k^q}{ \supalpha \sum_{k=m}^\infty \EE_\Palpha|X_k|^q / \ubar a_k^q } \right ) \right \rfloor \right \}.
 \end{equation}
 Appealing to \eqref{eq:uniform boundedness and integrability polynomial moment corollary} as well as \eqref{eq:lower-bounded-liminf-variance}, we have that $\ubar n_m \to \infty$ as $\mto$ and thus by monotonicity of $\ubar a_n / a_n$, we have that
 \begin{equation}
   \sup_\alphain (\ubar a_{n_m^\brackalpha} / a_{n_m^\brackalpha}) \leq \ubar a_{\ubar n_m} / a_{\ubar n_m} \to 0,
 \end{equation}
 which completes the proof.
\end{proof}

\subsection{Proof of \cref*{theorem:uniform-kmt-power-moments}}\label{proof:uniform-kmt-power-moments}

In several instances throughout the proof of \cref{theorem:uniform-kmt-power-moments}, we will apply a stochastic and uniform generalization of Kronecker's lemma \citep[Lemma 1]{waudby2024distribution} to show that certain sequences vanish uniformly. In order to state this generalized Kronecker's lemma, we must review the notions of uniform Cauchy sequences as well as sequences that are uniformly stochastically nonincreasing \citep[Definitions 2 \& 3]{waudby2024distribution}.

\begin{definition}[Uniform Cauchy sequences and stochastic nonincreasingness \citep{waudby2024distribution}]
  \label{definition:cauchy-sequences-nonincreasingness}
  Let $\Acal$ be an index set and $\infseqn{Y_n}$ a sequence of random variables defined on $\probspaces$. We say that $\infseqn{Y_n}$ is an \uline{$\Acal$-uniform Cauchy sequence} if for any $\eps > 0$,
  \begin{equation}\label{eq:cauchy}
    \lim_\mto \supalpha \Palpha \left ( \sup_{k,n \geq m} |Y_k - Y_n| \geq \eps \right ) = 0.
  \end{equation}
  Furthermore, we say that $\infseqn{Y_n}$ is \uline{$\Acal$-uniformly stochastically nonincreasing} if for every $\delta > 0$, there exists some $B_\delta > 0$ so that for every $n \geq 1$,
  \begin{equation}\label{eq:stochastically-nonincreasing}
    \supalpha \Palpha \left ( |Y_n| \geq B_\delta \right ) < \delta.
  \end{equation}
\end{definition}
Observe that when $\Acal = \{ \alphadot \}$ is a singleton, \eqref{eq:cauchy} and \eqref{eq:stochastically-nonincreasing} reduce to saying that $\infseqn{Y_n}$ is $P_\alphadot$-almost surely a Cauchy sequence and uniformly (in $n \in \NN$) bounded in $P_\alphadot$-probability, respectively. With these definitions in mind, we are ready to state the uniform Kronecker lemma of \citep[Lemma 1]{waudby2024distribution}.
\begin{lemma}[A stochastic and distribution-uniform Kronecker lemma \citep{waudby2024distribution}]\label{lemma:kronecker}
  Let $\infseqn{Z_n}$ be a sequence of random variables so that their partial sums $S_n := \sum_{i=1}^n Z_i$ form a uniform Cauchy sequence and which are uniformly stochastically nonincreasing as in \eqref{eq:cauchy} and \eqref{eq:stochastically-nonincreasing}. Let $\infseqn{b_n}$ be a positive, nondecreasing, and diverging sequence. Then
  \begin{equation}
    b_n^{-1} \sum_{i=1}^n b_i Z_i = \oAcalas(1).
  \end{equation}
\end{lemma}

\begin{remark}
In the setting of \citep{waudby2024distribution}, the measurable spaces $(\Omega_\alpha, \Fcal_\alpha)$ all coincide, and only the probability measures $P_\alpha$ vary with $\alpha$. However, this does not amount to any loss of generality, as Definition~\ref{definition:cauchy-sequences-nonincreasingness} and Lemma~\ref{lemma:kronecker} only depend on the laws $\mu_\alpha$ of $(Y_n)_{n=1}^\infty$ under $P_\alpha$. More formally, Lemma~\ref{lemma:kronecker} is readily deduced by applying \citep[Lemma 1]{waudby2024distribution} to the probability spaces $(\Omega, \Fcal, \mu_\alpha)$, where $\Omega = \mathbb R^{\mathbb N}$ and $\Fcal$ is its Borel $\sigma$-algebra, and the sequence in question is the canonical random element on $\Omega$.
\end{remark}

Before proceeding to the proof of \cref{theorem:uniform-kmt-power-moments}, we present a few results that are central to the intermediate arguments found in \cref{theorem:uniform-kmt-power-moments}.
In what follows, \cref{theorem:shao-einmahl-lifshits} allows for a uniform strong approximation with partial sums of Gaussians having truncated variances.

\begin{proposition}\label{theorem:shao-einmahl-lifshits}
  Let $\infseqn{X_n}$ be \iid{} mean-zero random variables on $\probspaces$.
  Suppose that the variance is uniformly positive: $\inf_\alphain \Var_\Palpha(X) > 0$ and that the $q^\tth$ moment is uniformly integrable for some $q > 2$:
  \begin{equation}
    \lim_\Kto \sup_\alphain \EE_\Palpha \left ( |X|^q  \1 \{ |X|^q \geq K \} \right )= 0.
  \end{equation}
  Then there exist constructions so that $\infseqn{\widetilde Y_n^\brackleq}$ are mean-zero independent Gaussians with variances given by
  \begin{equation}
    \Var_{\Palphatil} (\widetilde Y_k^\brackleq) = \Var_\Palpha \left ( X_k \1 \{ |X_k|^q \leq k \} \right )
  \end{equation}
  for every $k \in \NN$ and $\alphain$ so that
  \begin{equation}
    \sum_{k=1}^n (\widetilde X_k - \widetilde Y_k^\brackleq) = \oAcalas(n^{1/q}).
  \end{equation}
\end{proposition}

\begin{proof}[Proof of \cref{theorem:shao-einmahl-lifshits}]
  We first establish the following fact which will be used several times throughout the proof.
  Let $x$ be a real number. We aim to show that for any $q > 2$ and any integer $n \geq 1$,
  \begin{equation}\label{eq:lower-truncated-upper-bd}
    \sum_{k=1}^n \frac{|x| \1 \{ |x|^q \geq k \}}{k^{1/q}} \leq 2 \left ( |x|^q + 1  \right ).
  \end{equation}
  Indeed, suppose that $x > 0$ since otherwise the inequality holds trivially. Writing out the left-hand side of the desired inequality for any $x > 0$, we have that since $|x| \1 \{ |x|^q \geq k \} \geq 0 $ for all $n \leq k \leq \left \lceil |x|^q \right \rceil$ (if any), 
  \begin{align}
    \sum_{k=1}^n \frac{|x| \1 \{ |x|^q \geq k \}}{k^{1/q}} &\leq \sum_{k=1}^{\left \lceil |x|^q \right \rceil } \frac{|x|\1\{ |x|^q \geq k \}}{k^{1/q}}\\
    \verbose{
                                                           &\leq |x| \sum_{k=1}^{\left \lceil |x|^q \right \rceil } \frac{1}{k^{1/q}}\\
    }
                                                           &\leq |x|  \int_{0}^{\left \lceil |x|^q \right \rceil } k^{-1/q}dk \\
    \verbose{
                                                           &= |x| \left ( 1 + y^{-1/q + 1} \Bigm \vert_{y=1}^{\left \lceil |x|^q \right \rceil } \right )\\
                                                           &= |x| \frac{\left \lceil |x|^q \right \rceil }{\left \lceil |x|^q \right \rceil^{1/q} } \\
    }
    &\leq |x| \frac{1}{1 - 1/q} \frac{\left \lceil |x|^q \right \rceil }{(|x|^q)^{1/q}},
  \end{align} 
  which is at most $2 (|x|^q + 1)$ and hence we have \eqref{eq:lower-truncated-upper-bd}.

  Now, we re-write $\sum_{k=1}^n X_k$ and break it up into three terms:
  \begin{align}
    \sum_{k=1}^n X_k = &\underbrace{\sum_{k=1}^n \left ( X_k \1\{ |X_k| \leq k^{1/q} \} - \EE_P \left [ X_k \1\{ |X_k| \leq k^{1/q} \} \right ]  \right )}_{(i)}\\
    +\ &\underbrace{\sum_{k=1}^n X_k \1\{ |X_k| > k^{1/q} \}}_{(ii)} - \underbrace{\sum_{k=1}^n \EE_P \left [ X_k \1\{ |X_k| > k^{1/q} \} \right ]}_{(iii)}.
  \end{align}
  In what follows, we will first employ \cref{corollary:shao-strong-approx-uniform} to show that $(i)$ admits a uniform strong approximation so that $(i) - \sum_{k=1}^n \widetilde Y_k^\brackleq = \oAcalas(n^{1/q})$. We will then show that $(ii)$ and $(iii)$ are both $\oAcalas(n^{1/q})$ by first establishing that weighted versions of them are $\Acal$-uniformly Cauchy and $\Acal$-uniformly stochastically nonincreasing, to which we apply the $\Acal$-uniform Kronecker lemma (see \cref{lemma:kronecker}). 

  \paragraph{Analyzing the centered term $[ (i) - \sum_{k=1}^n \widetilde Y_k^\brackleq ]$}
  Throughout, let $p > q$ be arbitrary.
  We will show that the conditions of \cref{corollary:shao-strong-approx-uniform} are satisfied with the independent (but non-\iid{}) random variables
  \begin{equation}
   Z_k := X_k \1 \{ |X_k| \leq k^{1/q} \}, \quad k \in \NN
  \end{equation}
  but with the $q^\tth$ moment as used in \cref{corollary:shao-strong-approx-uniform} replaced by $p > q$ and with $a_k = k^{1/q}$.
  Note that while we are considering the $p^\tth$ moments of $Z_k$, the random variable $|X_k|$ is still truncated at $k^{1/q}$.
  Note that the centered moment $\EE|Z - \EE Z|^p$ of any random variable $Z$ is at most $2^p \EE|Z|^p$, and hence it suffices to consider the conditions of \cref{corollary:shao-strong-approx-uniform} but with $\EE_\Palpha |Z_k - \EE_\Palpha Z_k|^p$ replaced by $\EE_\Palpha |Z_k|^p$ everywhere. Indeed for any $m \geq 2$ and any $\alphain$, we have that
  \begin{align}
    \sum_{k=m}^\infty \frac{\EE_\Palpha \left \lvert X_k\1 \{ |X_k| \leq k^{1/q} \} \right \rvert^p}{a_k^p} &= \sum_{k=m}^\infty \frac{\EE_\Palpha \left ( \left \lvert X_k \right \rvert^p \1\{ |X_k| \leq k^{1/q} \} \right ) }{k^{p/q}}\\
          &= \sum_{k=m}^\infty \sum_{j=1}^k  \frac{\EE_\Palpha \left ( \left \lvert X_k \right \rvert^p \1 \{ (j-1) < |X_k|^q \leq j \} \right )}{k^{p/q}}\\
    \verbose{&= \sum_{k=1}^\infty \sum_{j=1}^k \1\{k\geq m\} \frac{\EE_\Palpha \left ( \left \lvert X_k \right \rvert^p \1 \{ (j-1) < |X_k|^q \leq j \} \right ) }{k^{p/q}}\\
    &= \sum_{j=m}^\infty \sum_{k=j}^\infty \frac{\EE_\Palpha \left ( \left \lvert X \right \rvert^p \1 \{ (j-1) < |X|^q \leq j \} \right ) }{k^{p/q}}\\
    }
                        &= \sum_{j=1}^\infty \EE_\Palpha \left ( \left \lvert X \right \rvert^p \1 \{ (j-1) < |X|^q \leq j \} \right ) \sum_{k=j\lor m}^\infty \frac{1}{k^{p/q}}\\
    &\leq \sum_{j=1}^\infty \EE_\Palpha \left ( \left \lvert X \right \rvert^p \1 \{ (j-1) < |X|^q \leq j \} \sum_{k=\left \lfloor |X|^q - 1 \right \rfloor\lor m}^\infty \frac{1}{k^{p/q}} \right ).
      \end{align}
      Therefore, there exists a constant $C_{p,q} > 0$ depending only on $ p> q > 2$ so that
      \begin{align}
\sum_{k=m}^\infty \frac{\EE_\Palpha \left \lvert X_k\1 \{ |X_k| \leq k^{1/q} \} \right \rvert^p}{a_k^p} &\leq \sum_{j=1}^\infty \EE_\Palpha\left ( \left \lvert X \right \rvert^p \1 \{ (j-1) < |X|^q \leq j \} C_{p,q}  \left (|X|^q \lor (m-1) \right )^{1-p/q}\right )\\
                                                                                                        &= C_{p,q} \cdot  \underbrace{\EE_\Palpha \left [|X|^p \left (|X|^q \lor (m-1) \right )^{1-p/q} \right ]}_{(\star)}.
  \end{align}
  Splitting $(\star)$ into the cases $|X|^q \leq \sqrt{m-1}$ and $|X|^q > \sqrt{m-1}$, we have that
  \begin{align}
    (\star) \leq\ &\EE_\Palpha \left [ |X|^q |X|^{p-q} (m-1)^{1-p/q} \1 \{ |X|^{p-q} \leq (m-1)^{\frac{p-q}{2q}} \} \right ] \\
                  &+ \EE_\Palpha \left [ |X|^q \1 \{ |X|^q > \sqrt{m-1} \} \right ] \\
    \leq\ & (m-1)^{\frac{1-p/q}{2}}\EE_\Palpha \left [ |X|^q \right ] + \EE_\Palpha \left [ |X|^q \1 \{ |X|^q > \sqrt{m-1} \} \right ],
  \end{align}
  noting that the right-hand side vanishes $\Acal$-uniformly as $\mto$.
  The same sum for $m = 1$ follows by combining the previous upper bound for $m \geq 2$ with the fact that $\EE_\Palpha |X_1\1 \{ |X_1| \leq 1 \}|^p \leq 1$.

  After invoking \cref{lemma:truncated-variance-from-uniform-integrability}, we have that the conditions of \cref{corollary:shao-strong-approx-uniform} are satisfied for the $p^\tth$ moment as applied to $X_k \1 \{ |X_k| \leq k^{1/q} \}$ and $a_k = k^{1/q}$ for each $k \in \NN$. Therefore, there exists a construction and independent mean-zero Gaussian random variables $\infseqn{\Ytil_n^\brackleq}$ with variances given by $\Var_\Palphatil(\Ytil_k^\brackleq) = \Var_\Palphatil(\Xtil_k \1 \{ |\Xtil_k| \leq k^{1/q} \})$ for each $k\in \NN$ so that
  \begin{equation}
   \sum_{k=1}^n \left ( \widetilde X_k \1\{ |\widetilde X_k| \leq k^{1/q} \} - \EE_\Palphatil \left [\widetilde X_k \1\{ |\widetilde X_k| \leq k^{1/q} \} \right ]  \right ) - \sum_{k=1}^n \widetilde Y_k^\brackleq = \oAcalas(n^{1/q}),
  \end{equation}
  which takes care of the first term in the decomposition of $\sum_{k=1}^n X_k$.

  \paragraph{Analyzing the sum of lower-truncated random variables in $(ii)$} 
  We are tasked with showing that a weighted variant of $(ii)$ is both uniformly Cauchy and uniformly stochastically nonincreasing. Beginning with the former, observe that for any $\alphain$ and any $\eps > 0$,
  \begin{align}
    &\Palpha \left ( \sup_{k, n \geq m} \left \lvert \sum_{i=1}^n \frac{X_i \1 \{ |X_i| > i^{1/q} \}}{i^{1/q}} - \sum_{i=1}^k \frac{X_i \1 \{ |X_i| > i^{1/q} \}}{i^{1/q}}  \right \rvert \geq \eps \right ) \\
    =\ &\Palpha \left ( \sup_{k \geq n \geq m} \left \lvert \sum_{i=n+1}^k \frac{X_i \1 \{ |X_i| > i^{1/q} \}}{i^{1/q}} \right \rvert \geq \eps \right ) \\
    \verbose{
    \leq\ &\Palpha \left ( \exists k \geq m : |X_k|  > k^{1/q} \right ) \\
    }
    \leq\ & \sum_{k=m}^\infty \Palpha \left ( |X|^q > k \right ) \\
    \leq\ & \EE_\Palpha \left ( |X|^q \1 \{ |X|^q > m \} \right ),
  \end{align}
  and hence after taking a supremum over $\alphain$ and the limit as $\mto$, we have that the $i^{-1/q}$-weighted version of $(ii)$ is a uniform Cauchy sequence. Turning now to showing that this version of $(ii)$ is uniformly stochastically nonincreasing, observe that for any $\alphain$, $B > 0$, and $n \geq 1$, we have
  \begin{align}
    \Palpha \left ( \left \lvert \sum_{i=1}^n \frac{X_i\1 \{ |X_i| > i^{1/q} \}}{i^{1/q}} \right \rvert \geq B \right ) &\leq \EE_\Palpha \left \lvert \sum_{i=1}^n \frac{(X_i/B) \1 \{ |X_i| > i^{1/q} \}}{i^{1/q}} \right \rvert \\
    \verbose{
    &\leq \frac{1}{B} \EE_\Palpha \left (\sum_{i=1}^n \frac{ |X| \1 \{ |X| > i^{1/q} \}}{i^{1/q}} \right ) \\
    }
    &\leq \frac{2}{B} \left [ \EE_\Palpha \left ( |X|^q \right ) + 1 \right ].
  \end{align}
  where the final inequality follows from an application of \eqref{eq:lower-truncated-upper-bd}. Since the $q^\tth$ moment is $\Acal$-uniformly integrable by assumption, the right-hand side is $\Acal$-uniformly bounded and hence we can find $B > 0$ sufficiently large so that the left-hand side is as small as desired uniformly in $\Acal$. Combining this with the fact that the $i^{-1/q}$-weighted variant of $(ii)$ is uniformly Cauchy and invoking \cref{lemma:kronecker}, we have that
  \begin{equation}
    (ii) = \oAcalas(n^{1/q}).
  \end{equation}

  \paragraph{Analyzing the expectation of $(ii)$ in $(iii)$}
  Similarly to the analysis of $(ii)$, we are tasked with showing that a weighted version of $(iii)$ is both uniformly Cauchy and stochastically nonincreasing. Beginning with the former property, we have for any $\alphain$ and $\eps > 0$,
  \begin{align}
    &\Palpha \left ( \supknm \left \lvert \sum_{i=1}^k \frac{\EE_\Palpha \left [ X_i \1 \{ |X_i| > i^{1/q} \} \right ]}{i^{1/q}} - \sum_{i=1}^n \frac{ \EE_\Palpha\left [ X_i \1 \{ |X_i| > i^{1/q} \} \right ]}{i^{1/q}} \right \rvert \geq \eps \right ) \\
    \verbose{
    =\ & \Palpha \left ( \sup_{k \geq n\geq m} \left \lvert \sum_{i=n+1}^k \frac{ \EE_\Palpha\left [ X \1 \{ |X|^q > i \} \right ]}{i^{1/q}} \right \rvert \geq \eps \right )\\
    \leq\ & \Palpha \left ( \sup_{k \geq n\geq m} \sum_{i=n+1}^k \left \lvert \frac{ \EE_\Palpha\left [ X \1 \{ |X|^q > i \} \right ]}{i^{1/q}} \right \rvert \geq \eps \right )\\
    }
    \leq\ & \Palpha \left ( \sup_{k \geq n\geq m} \sum_{i=n+1}^k\frac{ \EE_\Palpha\left [ |X| \1 \{ |X|^q > i \} \right ]}{i^{1/q}} \geq \eps \right )\\
    =\ & \1 \left \{ \sup_{k \geq n\geq m} \EE_\Palpha \left ( \sum_{i=n+1}^k\frac{ |X| \1 \{ |X|^q > i \} \1 \{ |X|^q > m \} }{i^{1/q}} \right ) \geq \eps \right \},
         \end{align}
         where we have used the fact that for every $i > n \geq m$, we have $\1 \{ |X|^q > i \} = \1 \{ |X|^q > i \}\1 \{ |X|^q > m \}$. Continuing from the above,
         \begin{align}
           &\1 \left \{ \sup_{k \geq n\geq m} \EE_\Palpha \left ( \sum_{i=n+1}^k\frac{ |X| \1 \{ |X|^q > i \} \1 \{ |X|^q > m \} }{i^{1/q}} \right ) \geq \eps \right \}\\
           \verbose{
           =\ &\1 \left \{ \sup_{k \geq n\geq m} \EE_\Palpha \left (\1 \{ |X|^q > m \} \cdot \sum_{i=n+1}^k\frac{ |X| \1 \{ |X|^q > i \}  }{i^{1/q}} \right ) \geq \eps \right \}\\ 
           }
           \leq\ & \1 \left \{ \sup_{k \geq m} \EE_\Palpha \left ( \1 \{ |X|^q > m \}\cdot \sum_{i=1}^k\frac{ |X| \1 \{ |X|^q > i \}  }{i^{1/q}} \right ) \geq \eps \right \}\\ 
           \leq\ & \1 \left \{ 2\EE_\Palpha \left ( \1 \{ |X|^q > m \}\cdot \left [ |X|^q + 1 \right ]  \right ) \geq \eps \right \}\\ 
           \leq\ & \1 \left \{ 4 \EE_\Palpha \left ( |X|^q\1 \{ |X|^q > m \} \right ) \geq \eps \right \},
         \end{align}
         where the second-last inequality follows from another application of \eqref{eq:lower-truncated-upper-bd}. Notice that the above
         vanishes uniformly in $\alphain$ as $\mto$ by uniform integrability of the $q^\tth$ moment and thus the desired partial sums form a uniform Cauchy sequence. It remains to check whether $\sum_{i=1}^k \EE_\Palpha \left [ X_i \1 \{ |X_i| > i^{1/q} \} \right ] / i^{1/q}$ is $\Acal$-uniformly bounded in probability. Indeed, this is simpler than in the case of $(ii)$ since $\sum_{i=1}^k \EE_\Palpha \left [ X_i \1 \{ |X_i| > i^{1/q} \} / i^{1/q} \right ]$ is deterministic, and thus it suffices to show that there exists $B > 0$ sufficiently large so that for any $n \geq 1$,
         \begin{equation}
           \supalpha \1 \left \{  \sum_{i=1}^n \EE_\Palpha \left [ \frac{|X_i| \1 \{ |X_i| > i^{1/q} \}}{i^{1/q}} \right ] > B\right \}
         \end{equation}
         is zero, or in other words, that $\sup_\alphain \sum_{i=1}^n \EE_\Palpha \left [ |X_i| \1 \{ |X_i| > i^{1/q} \} / i^{1/q} \right ] < \infty$. Indeed, by the inequality in \eqref{eq:lower-truncated-upper-bd}, we have that
         \begin{align}
           \sup_\alphain \sum_{i=1}^k \EE_\Palpha \left [ \frac{|X_i| \1 \{ |X_i| > i^{1/q} \}}{i^{1/q}} \right ] &\leq 2 \supalpha \EE_\Palpha \left ( |X|^q + 1 \right ),
         \end{align}
         which is bounded by uniform integrability of the $q^\tth$ moment.
  Finally, applying \cref{lemma:kronecker}, we have that $(iii) = \oAcalas(n^{1/q})$
  which completes the analysis of $(iii)$ and hence the proof of \cref{theorem:shao-einmahl-lifshits}.
\end{proof}

The following proposition demonstrates that under a uniform integrability and uniformly positive variance assumption, only finitely many of the truncated variances in \cref{theorem:shao-einmahl-lifshits} can be zero.
\begin{proposition}\label{lemma:truncated-variance-from-uniform-integrability}
  Let $X$ be a mean-zero random variable with an $\Acal$-uniformly integrable second moment and an $\Acal$-uniformly positive variance i.e.
  \begin{equation}
    \lim_\Kto \supalpha \EE_\Palpha \left [ |X|^2 \1 \{ |X|^2 \geq K \} \right ] = 0\quad\text{and}\quad \inf_\alphain \Var_\Palpha (X) > 0.
  \end{equation}
  Then, the upper-truncated version of $X$ has a uniformly positive variance, i.e.
  \begin{equation}
    \liminf_\Kto \inf_\alphain \Var_\Palpha (X \1 \{ |X| \leq K \}) > 0.
  \end{equation}
\end{proposition}
\begin{proof} Writing out the variance of $X \1 \{ |X| \leq K \}$ under $\Palpha$ for any $K > 0$ and $\alphain$, we have
  \begin{align}
    \Var_\Palpha ( X \1 \{ |X| \leq K \}) &= \EE_\Palpha \left [ X^2 \1 \{ |X| \leq K \} \right ] - \left ( \EE_\Palpha \left [ X \1 \{ |X| \leq K \} \right ] \right )^2 \\
                                          &= \EE_\Palpha \left [ X^2 - X^2 \1 \{ |X| > K \} \right ] - \left ( \EE_\Palpha \left [ X \1 \{ |X| \leq K \} \right ] \right )^2 \\
                                          &= \Var_\Palpha (X) - \EE_\Palpha \left [ X^2 \1 \{ |X| > K \} \right ] - \left ( \EE_\Palpha \left [ X \1 \{ |X| > K \} \right ] \right )^2 \label{eq:proof uniformly positive variance using mean zero} \\
    &\geq \Var_\Palpha (X) - \EE_\Palpha \left [ X^2 \1 \{ |X| > K \} \right ] - \left ( \EE_\Palpha \left [ |X| \1 \{ |X| > K \} \right ] \right )^2,
  \end{align}
  where in \eqref{eq:proof uniformly positive variance using mean zero} we used the fact that $X$ is mean-zero, and in the final inequality we used the fact that $| \EE_\Palpha [X \1 \{ |X| > K \}] | \leq \EE_\Palpha [|X| \1 \{ |X| > K \}]$ by Jensen's inequality. Taking an infimum over $\alphain$ in the left-hand side of the above, we have that
  \begin{align}
    &\inf_\alphain \Var_\Palpha \left ( X \1 \{ |X| \leq K \} \right )\\
    \geq\ &\inf_\alphain \Var_\Palpha (X) - \sup_\alphain \EE_\Palpha \left [ X^2 \1 \{ |X| > K \} \right ] - \left ( \sup_\alphain \EE_\Palpha \left [ |X| \1 \{  |X| > K\} \right ] \right )^2.
  \end{align}
  Appealing to $\Acal$-uniform integrability of the second (and hence first) moment and taking a limit infimum as $\Kto$, the desired result follows.
\end{proof}
The following lemma establishes a bound on a weighted series of variances of certain Gaussian differences in terms of a uniform integrability tail.

\begin{lemma}\label{lemma:partial-sums-of-variance-of-diff}
  Let $\infseqn{X_n}$ be \iid{} random variables on $\probspace$ with mean zero, variance $\sigma^2$, and suppose that $\EE_P|X|^q < \infty$ for $q > 2$. Let $\infseqn{\widetilde Y_n}$ be independent Gaussian random variables on $\probspacetilde$ with mean zero and
  \begin{equation}
    \widetilde \sigma_k^2 := \Var_\Ptil(\widetilde Y_k) = \Var_P(X \1 \{ |X| \leq k^{1/q} \} ).
  \end{equation}
  Let $\widehat Y_k := (\sigma / \widetilde \sigma_k)\cdot \widetilde Y_k \1 \{ \widetilde \sigma_k > 0 \} + \sqrt{\Var_P(X)}Z_k\1 \{ \widetilde \sigma_k = 0 \}$ for each $k \in \NN$, where $\infseqn{Z_n}$ is an \iid{} sequence of standard Gaussians. We have that for any $m \geq 1$,
  \begin{equation}
    \sum_{k=m}^\infty \frac{\Var_\Ptil(\widetilde Y_k - \widehat Y_k)}{k^{2/q}} \leq \sum_{k=m}^\infty \frac{\Var_P(X) \1 \{ \widetilde \sigma_k = 0 \}}{k^{2/q}} + 4 C_q \EE_P \left ( |X|^q \1 \{ |X|^q > m \}  \right ),
  \end{equation}
  where $C_q$ is a constant depending only on $q$.
\end{lemma}

\begin{proof}
  To begin, fix $k \in \NN$. Note that if $\widetilde \sigma_k > 0$, then since $\widehat Y_k - \widetilde Y_k = \widetilde Y_k(\sigma / \widetilde \sigma_k - 1)$, we have that
  \begin{align}
    \Var_\Ptil(\widehat Y_k - \widetilde Y_k) &= (\sigma / \widetilde \sigma_k - 1)^2 \widetilde \sigma_k^2 \1 \{ \widetilde \sigma_k > 0 \} \\
    \verbose{&= (\sigma-\widetilde \sigma_k)^2 \\
    &= \sigma^2 - 2\sigma \widetilde \sigma_k + \widetilde \sigma_k^2\\
    }
                                 &\leq \sigma^2 - \widetilde \sigma_k^2,
  \end{align}
  where the inequality follows from the fact that $0 < \widetilde \sigma_k^2 \leq \sigma^2$ for all $k \in \NN$.
  Now, observe that we can use the fact that $\EE_P(X) = 0$ to further upper-bound the above quantity as
  \begin{align}
    \sigma^2-\widetilde \sigma_k^2 &\equiv \Var_P(X)- \Var_P(X \1 \{ |X| \leq k^{1/q} \} ) \\
    &=  \EE_P X^2 - \EE_P  ( X^2 \1 \{ |X| \leq k^{1/q} \}  ) + \left [ \EE_P  ( X \1 \{ |X| \leq k^{1/q} \}  ) \right ]^2\\
    \verbose{
    &= \EE_P  ( X^2 \1 \{ |X| > k^{1/q} \}  ) + \left [ \EE_P  ( X \1 \{ |X| \leq k^{1/q} \}  ) \right ]^2\\
    &= \EE_P X^2 \1 \{ |X| > k^{1/q} \} +  \left [ -\EE_P ( X \1 \{ |X| > k^{1/q} \} )  \right ]^2\\
    }
    &\leq \EE_P ( X^2 \1 \{ |X| > k^{1/q} \} ) + \EE_P  ( X^2 \1 \{ |X| > k^{1/q} \} )\\
    &= 2\EE_P X^2 \1 \{ |X| > k^{1/q} \}.
  \end{align}
  On the other hand, if $\widetilde \sigma_k = 0$, we have that
  \begin{equation}
    \Var_\Ptil(\widehat Y_k - \widetilde Y_k) = \Var_P(X).
  \end{equation}
  Turning to the quantity that we aim to ultimately upper bound $\sum_{k=m}^\infty \Var_\Ptil(\widetilde Y_k - \widehat Y_k) / k^{2/q}$, we have that
  \begin{align}
    \sum_{k=m}^\infty \frac{\Var_\Ptil(\widetilde Y_k - \widehat Y_k)}{k^{2/q}} \leq \sum_{k=m}^\infty \frac{\Var_P(X)\1 \{ \widetilde \sigma_k = 0 \}}{k^{2/q}} + \underbrace{\sum_{k=m}^\infty \frac{2\EE_P X^2 \1 \{ |X| > k^{1/q} \}}{k^{2/q}}}_{(\star)}.
  \end{align}
  By Tonelli's theorem, we can re-write $(\star)$ as
  \begin{align}
    (\star)  = 2\EE_P \left ( X^2 \sum_{k=m}^\infty \frac{ \1 \{ |X| > k^{1/q} \}}{k^{2/q}} \right ).
  \end{align}
  Using the identity $\sum_{k=1}^\infty \sum_{j=k}^\infty a_{k,j} = \sum_{j=1}^\infty \sum_{k=1}^j a_{k,j}$ for any $a_{k,j} > 0$; $k,j \in \NN$, we have that
  \begin{align}
    (\star) &= 2\EE_P \left ( X^2 \sum_{k=m}^\infty \sum_{j=k}^\infty  \frac{ \1 \{ j < |X|^q \leq j + 1 \}}{k^{2/q}} \right )\\
    \verbose{
    &= 2\EE_P \left ( X^2 \sum_{k=1}^\infty \sum_{j=k}^\infty \1 \{ k \geq m \} \cdot \frac{ \1 \{ j < |X|^q \leq j + 1 \}}{k^{2/q}} \right )\\
    &= 2\EE_P \left ( X^2 \sum_{j=1}^\infty \sum_{k=1}^j \1 \{ k \geq m \} \cdot \frac{ \1 \{ j < |X|^q \leq j + 1 \}}{k^{2/q}} \right )\\
    &= 2\EE_P \left ( X^2 \sum_{j=m}^\infty \sum_{k=m}^j \frac{ \1 \{ j < |X|^q \leq j + 1 \}}{k^{2/q}} \right )\\
    }
    &= 2\EE_P \left ( X^2 \sum_{j=m}^\infty \sum_{k=m}^j \frac{ \1 \{ j < |X|^q \leq j + 1 \}}{k^{2/q}} \right )\\
                                                                                &\leq 2\EE_P \left ( X^2 \sum_{j=m}^\infty \1 \{ j < |X|^q \leq j + 1 \} \sum_{k=1}^j \frac{1}{k^{2/q}} \right ).
\end{align}
Focusing on the partial sum $\sum_{k=1}^j k^{-2/q}$, we have that there exists a constant $C_q > 0$ depending only on $q > 2$ so that $\sum_{k=1}^j k^{-2/q} \leq C_q (j+1)^{1-2/q}$, and hence the above can be further upper-bounded as
\begin{align}
    (\star)&\leq 2\EE_P \left ( X^2 \sum_{j=m}^\infty \1 \{ j < |X|^q \leq j + 1 \} \sum_{k=1}^j \frac{1}{k^{2/q}} \right )\\
  \verbose{
                                                                 &\leq 2 C_q \EE_P \left ( X^2 \sum_{j=m}^\infty \1 \{ j < |X|^q \leq j + 1 \} (j+1)^{1-2/q} \right )\\
                                                                 &\leq 2 C_q \EE_P \left (  \sum_{j=m}^\infty \cancel{(j+1)^{2/q}} \1 \{ j < |X|^q \leq j + 1 \} \frac{(j+1)}{\cancel{(j+1)^{2/q}}} \right )\\
                                                                 &\leq 2 C_q \EE_P \left ( \left [ |X|^q + 1 \right ] \sum_{j=m}^\infty  \1 \{ j < |X|^q \leq j + 1 \} \right )\\
  }
                                                                 &\leq 4 C_q \EE_P \left (  |X|^q \sum_{j=m}^\infty  \1 \{ j < |X|^q \leq j + 1 \} \right )\\
                                                                 &= 4 C_q \EE_P \left (  |X|^q \1 \{ |X|^q > m \} \right ).
  \end{align}
  which completes the proof of \cref{lemma:partial-sums-of-variance-of-diff}.
\end{proof}

With \cref{definition:cauchy-sequences-nonincreasingness,lemma:kronecker,theorem:shao-einmahl-lifshits,lemma:truncated-variance-from-uniform-integrability,lemma:partial-sums-of-variance-of-diff} in mind, we are ready to prove \cref{theorem:uniform-kmt-power-moments}. Our proof structure borrows elements from the proofs of the pointwise results of \citet*{komlos1976approximation} and \citet{lifshits2000lecture}, as well as from the proofs of uniform strong laws of large numbers \citep{waudby2024distribution}.

\begin{proof}[Proof of \cref{theorem:uniform-kmt-power-moments}]
  We begin by showing that uniform integrability of the $q^\tth$ moment is sufficient for uniform strong approximation at the rate of $\oAcalas(n^{1/q})$ and later demonstrate it is also necessary. The latter half of the proof follows similarly to the necessity half of the proof of \cref{theorem:uniform-kmt-exponential}. Applying \cref{theorem:shao-einmahl-lifshits}, we have that there exist mean-zero independent Gaussians $\infseqn{\widetilde Y_n^\brackleq}$ with variances given by $\widetilde \sigma_k^2 := \Var ( \widetilde Y_k^\brackleq ) = \Var (X_k \1 \{ |X_k| \leq k^{1/q} \})$ so that
  \begin{equation}
    \sum_{k=1}^n  \widetilde X_k - \sum_{k=1}^n \widetilde Y_k^\brackleq = \oAcalas(n^{1/q})
  \end{equation}
  Now, define the sequence $\infseqn{\widehat Y_n}$ given by
  \begin{equation}
    \widehat Y_n := \sqrt{\frac{\Var(X)}{\widetilde \sigma_n^2}} \cdot \widetilde Y_n^\brackleq \1 \{ \widetilde \sigma_n > 0 \} + \sqrt{\Var(X)}Z_n \1 \{ \widetilde \sigma_n = 0 \}
  \end{equation}
  where $\infseqn{Z_n}$ are \iid{} standard Gaussians independent of all else.
  Note that the marginal distribution of $\infseqn{\widehat Y_n}$ is a sequence of \iid{} mean-zero Gaussian random variables with variances $\Var_\Palphatil(\widehat Y_n) = \Var_\Palpha(X)$ for each $n \in \NN$ and so it suffices to show that
  \begin{equation}\label{eq:diff-vanishing-desideratum}
    \sum_{k=1}^n \widehat Y_k - \sum_{k=1}^n \widetilde Y_k^\brackleq = \oAcalas(n^{1/q}).
  \end{equation}
  Indeed, we will achieve this by applying \cref{lemma:kronecker} to the random variables $(\widehat Y_k - \widetilde Y_k^\brackleq) / k^{1/q}$. That is, we will show that the sequence formed by
  \begin{equation}
    S_n := \sum_{k=1}^n \frac{(\widehat Y_k - \widetilde Y_k^\brackleq)}{k^{1/q}}
  \end{equation}
  is both $\Acal$-uniformly Cauchy and $\Acal$-uniformly stochastically nonincreasing from which the desired result in \eqref{eq:diff-vanishing-desideratum} will follow.

  \paragraph{Showing that $S_n$ is $\Acal$-uniformly Cauchy}
  First, notice that $\Var [ (\widehat Y_k - \widetilde Y_k^\brackleq) / k^{1/q} ] = \Var ( \widehat Y_k - \widetilde Y_k^\brackleq ) / k^{2/q}$
  and hence by \cref{lemma:partial-sums-of-variance-of-diff}, we have that for every $\alphain$ and $m \geq 1$,
  \begin{equation}\label{eq:bound with infinite sum and uniform integrability term}
    \sum_{k=m}^\infty \Var_\Palphatil [ (\widehat Y_k - \widetilde Y_k^\brackleq) / k^{1/q} ] \leq \sum_{k=m}^\infty \frac{\Var_\Palpha(X) \1 \{ \widetilde \sigma_k = 0 \}}{k^{2/q}} + 4 C_q \EE_\Palpha \left ( |X|^q \1 \{ |X|^q > m \}  \right ),
  \end{equation}
  where $C_q > 0$ depends only on $q$ and not on $\alphain$. Notice that by \cref{lemma:truncated-variance-from-uniform-integrability}, we have that for $m$ sufficiently large, all of the summands in the first term of \eqref{eq:bound with infinite sum and uniform integrability term} are zero for every $\alphain$.
  Indeed, taking a supremum over $\alphain$ and the limit as $\mto$, we apply
   the uniform Khintchine--Kolmogorov convergence theorem \citep[Theorem 3]{waudby2024distribution} and conclude that $S_n \equiv \sum_{k=1}^n (\widehat Y_k - \widetilde Y_k^\brackleq) / k^{1/q}$ is $\Acal$-uniformly Cauchy on $\probspacestilde$.

   \paragraph{Showing that $S_n$ is $\Acal$-uniformly stochastically nonincreasing}
   Let $\delta > 0$. We need to show that there exists some $B_\delta > 0$ so that for every $n \geq 1$,
   \begin{equation}
     \sup_\alphain \Palphatil \left ( \left \lvert \sum_{k=1}^n \frac{\widehat Y_k - \widetilde Y_k^\brackleq }{ k^{1/q}} \right \rvert \geq B_\delta \right ) \leq \delta.
   \end{equation}
   By Kolmogorov's inequality \citep[Theorem 22.4]{billingsley1995probability}, we have that for any $B > 0$ and any $\alphain$,
   \begin{equation}
     \Palphatil \left ( \left \lvert \sum_{k=1}^n \frac{\widehat Y_k - \widetilde Y_k^\brackleq}{k^{1/q}} \right \rvert \geq B \right ) \leq \frac{1}{B^2} \sum_{k=1}^n \frac{\Var_\Palphatil(\widehat Y_k - \widetilde Y_k^\brackleq)}{k^{2/q}}.
   \end{equation}
   Applying \cref{lemma:partial-sums-of-variance-of-diff} once again but with $m = 1$, we have that for any $n \geq 1$ and $B > 0$,
   \begin{equation}
     \sup_\alphain \Palphatil \left ( \left \lvert\sum_{k=1}^n  \frac{\widehat Y_k - \widetilde Y_k^\brackleq}{k^{1/q}} \right \rvert \geq B \right ) \leq \frac{1}{B^2} \sup_\alphain \sum_{k=1}^\infty \frac{\Var_\Palpha(X) \1 \{ \widetilde \sigma_k = 0 \}}{k^{2/q}} + \frac{C_q }{B^2} \sup_\alphain \EE_\Palpha |X|^q.
   \end{equation}
   The two terms on the right-hand side are finite by \cref{lemma:truncated-variance-from-uniform-integrability} and the uniform integrability assumption, respectively. Therefore, we can always find some $B_\delta > 0$ so that the right-hand side is smaller than $\delta$, and hence $S_n$ is $\Acal$-uniformly stochastically nonincreasing.

   Finally, using the fact that $S_n$ is both uniformly Cauchy and stochastically nonincreasing, we apply \cref{lemma:kronecker} to conclude that
  \begin{equation}
    \sum_{k=1}^n (\widehat Y_k - \widetilde Y_k^\brackleq) = \oAcalas(n^{1/q}),
  \end{equation}
  which completes the sufficiency half of the proof of \cref{theorem:uniform-kmt-power-moments}.

  Let us now move on to the proof of necessity. Suppose that the $q^\tth$ moment of $X$ is not $\Acal$-uniformly integrable, i.e.
  \begin{equation}
    \lim_\Kto \sup_\alphain \EE_\Palpha \left [ |X|^q \1 \{ |X|^q \geq K \} \right ] > 0,
  \end{equation}
  and we are tasked with showing that there exists some $\eps' > 0$ so that for any collection of constructions $\probspacealphatil_\alphain$,
  \begin{equation}\label{eq:uniform kmt power moments necessity goal}
    \lim_\mto \sup_\alphain \Palphatil \left ( \supkm \frac{|\Lambdatil_k|}{k^{1/q}} \geq \eps' \right ) > 0.
  \end{equation}
  Indeed, using a similar series of arguments as in the proof of necessity for \cref{theorem:uniform-kmt-exponential}, we have that for any collection of constructions,
  \begin{equation}
    \underbrace{\supalpha \Palpha \left ( \sup_{k \geq m}\frac{| X_k |}{k^{1/q}}\geq  1 \right )}_{\gamma_X(m)}  \leq \underbrace{2\supalpha \Palphatil \left ( \sup_{k \geq m-1} \frac{|\Lambdatil_k|}{k^{1/q}}  \geq 1/4 \right )}_{\gamma_{\Lambda}(m)} + \underbrace{\supalpha \Palphatil \left ( \supkm \frac{\lvert \widetilde Y_k \rvert}{k^{1/q}}  \geq  1/2 \right )}_{\gamma_Y(m)},
  \end{equation}
  and as in that proof, we have that $\gamma_Y(m) \to 0$ so it suffices to show that $\gamma_X(m)$ does not vanish as $\mto$, from which we can deduce that \eqref{eq:uniform kmt power moments necessity goal} holds with $\eps' = 1/4$. Indeed, by the uniform second Borel-Cantelli lemma, it suffices to show that
  \begin{equation}
    \lim_\mto \sup_\alphain \sum_{k=m}^\infty \Palpha \left ( |X| \geq k^{1/q} \right ) > 0,
  \end{equation}
  and by \citep[Theorem 2.1]{hu2017note} (see also \citep[Lemma 10]{waudby2024distribution}), the above holds if and only if the $q^\tth$ moment of $X$ is not $\Acal$-uniformly integrable. This completes the proof of necessity and hence of \cref{theorem:uniform-kmt-power-moments} altogether.
\end{proof}

\section{Summary \& future work}

This paper gave matching necessary and sufficient conditions for strong Gaussian approximations to hold uniformly in a collection of probability spaces under assumptions concerning both exponential and power moments, thereby generalizing some of the classical approximations of \citet*{komlos1975approximation,komlos1976approximation} for a single probability space. Along the way, we provided time-uniform concentration inequalities for differences between partial sums and their strongly approximated Gaussian sums, and these inequalities contain only explicit or universal constants, ultimately shedding nonasymptotic light on the problem of strong Gaussian approximation.

While we gave a comprehensive treatment of the problem under finiteness of exponential moments and $q^\tth$ power moments strictly larger than $2$, we did not consider the case where $q = 2$. Since solutions to the analogous problem in the pointwise setting considered by \citet{strassen1964invariance} rely on the Skorokhod embedding scheme (rather than dyadic arguments found in \citet{komlos1975approximation,komlos1976approximation} and our work alike), we anticipate that uniform strong approximations for only $q = 2$ finite (or uniformly integrable) moments will require the development of a different set of technical tools. We intend to pursue this thread in future work.

%% file: appendix.tex
\section{Sakhanenko regularity from sub-Gaussianity}\label{section:auxiliary lemmas - sakhanenko regularity from subgaussianity}

\begin{lemma}[Uniformly sub-Gaussian collections are Sakhanenko regular]\label{lemma:sub-gaussian-are-sakhanenko-regular}
  Fix $\sigma > 0$ and let $\Acal$ be an index set so that $X$ is $\Acal$-uniformly $\sigma$-sub-Gaussian with a variance that is $\Acal$-uniformly lower-bounded by $\ubar \sigma^2 > 0$. Then, $X$ is $(\ubar \lambda, \Acal)$-Sakhanenko regular where $\ubar \lambda$ is given by
  \begin{equation}\label{eq:sakhanenko-parameter-sub-gaussian-lemma}
    \ubar \lambda := \frac{\ubar \sigma^2}{8e \sqrt{3}\sigma^3}.
  \end{equation}
\end{lemma}
\begin{proof}
  Let $\ubar \sigma^2 > 0$ be the uniform lower bound on the variance: $\inf_\alphain \Var_\Palpha(X) \geq \ubar \sigma^2$.
  By H\"older's inequality and sub-Gaussianity, we have that for any $\alphain$ and any $\lambda \in \RR$,
  \begin{align}
    \EE_\Palpha \left ( |X|^3 \exp \left \{ \lambda |X| \right \} \right ) &\leq \left ( \EE_\Palpha|X|^6 \right )^{1/2} \left ( \EE_\Palpha \exp \{ 2\lambda |X| \} \right )^{1/2}\\
                                                                     &\leq \left ( \EE_\Palpha|X|^6 \right )^{1/2} \left ( 2 \exp \{ 4 \sigma^2 \lambda^2 / 2 \} \right )^{1/2}\\
    &= \left ( \EE_\Palpha|X|^6 \right )^{1/2} \cdot \sqrt{2} \exp \{ \sigma^2 \lambda^2 \}.
  \end{align}
  Using the inequality $\EE_P|Y|^q \leq q 2^{q/2} \sigma^q \Gamma(q/2)$ for $q \geq 1$ and $\sigma$-sub-Gaussian $Y$, we have that $\EE_\Palpha |X|^6 \leq 96 \sigma^6$, and hence
  \begin{equation}
    \ubar \lambda \EE_\Palpha \left ( |X|^3 \exp \left \{ \ubar \lambda |X| \right \} \right ) \leq \ubar \lambda 8\sqrt{3} \sigma^3  \cdot \exp \{ \sigma^2 \ubar \lambda^2 \},
  \end{equation}
  which is at most $\ubar \sigma^2$ if $\ubar \lambda$ is given as in \eqref{eq:sakhanenko-parameter-sub-gaussian-lemma}. This completes the proof.
\end{proof}